\documentclass[10pt]{amsart}

\usepackage{amssymb}
\usepackage{amsbsy}
\usepackage{amscd}
\usepackage[mathscr]{eucal}
\usepackage{verbatim}
\usepackage{version}
\usepackage{color}
\usepackage[matrix,arrow]{xy}
\usepackage{graphicx}

\def\cal{\mathcal}
\def\Bbb{\mathbb}

\newenvironment{NB}{
\color{red}{\bf NB}. \footnotesize 
}{}
\excludeversion{NB}
\newenvironment{NB2}{
\color{blue}{\bf NB}. \footnotesize
}{}

\newcommand{\simto}{\xrightarrow{\ \sim\ }}

\newcommand{\GL}  {\operatorname{GL}}

\newcommand{\Pic}{\operatorname{Pic}}

\newcommand{\ch}{\operatorname{ch}}

\newcommand{\Coh}{\operatorname{Coh}}

\newcommand{\Ext}{\operatorname{Ext}}
\newcommand{\Hom}{\operatorname{Hom}}

\newcommand{\rk}{\operatorname{rk}}

\newcommand{\NS}{\operatorname{NS}}
\newcommand{\coker}{\operatorname{coker}}
\newcommand{\td}{\operatorname{td}}

\newcommand{\Amp}{\operatorname{Amp}}

\newcommand{\alg}{\operatorname{alg}}

\newcommand{\Stab}{\operatorname{Stab}}

\def\Stabd{\mathop{\Stab^{\dagger}}\nolimits}

\font\b=cmr10 scaled \magstep5
\def\bigzerou{\smash{\lower1.7ex\hbox{\b 0}}}

\numberwithin{equation}{section}
\theoremstyle{plain}
 \newtheorem{thm}{Theorem}[section]
 \newtheorem{lem}[thm]{Lemma}
 \newtheorem{prop}[thm]{Proposition}
 \newtheorem{cor}[thm]{Corollary}
 
 \newtheorem{prob}[thm]{Problem}
\theoremstyle{definition}
 \newtheorem{defn}[thm]{Definition}
 
\theoremstyle{remark}
 \newtheorem{rem}[thm]{Remark}
 \newtheorem{ex}[thm]{Example}
 
\begin{document}

\title{Fourier-Mukai transforms by a family of stable sheaves and preservation of stability}
\author{K\={o}ta Yoshioka}
\address{Department of Mathematics, Faculty of Science,
Kobe University,
Kobe, 657, Japan
}
\email{yoshioka@math.kobe-u.ac.jp}

\thanks{
The author is supported by the Grant-in-aid for 
Scientific Research (No. 21H04429, 22H00094, 23K03053, 26K06742), JSPS}
\keywords{K3 surfaces, Bridgeland stability, Fourier-Mukai transforms}
%\subjclass[2010]{Primary 14D20}

\begin{abstract}
For a Fourier-Mukai transform whose kernel is a family of Gieseker stable sheaves, 
we shall give a condition for the preservation of Gieseker stability of sheaves. 
\end{abstract}

\maketitle
%\tableofcontents

\renewcommand{\thefootnote}{\fnsymbol{footnote}}
\footnote[0]{2010 \textit{Mathematics Subject Classification}. 
Primary 14D20.}

\section{Introduction}

Let $X$ be a K3 surface and $H$ an ample divisor on $X$.
Fourier-Mukai transforms are equivalences of the derived categories of 
coherent sheaves, and there are many Fourier-Mukai transforms on K3 surfaces.
In particular a 2-dimensional fine moduli space $X'$ of stable sheaves on $X$ defines a Fourier-Mukai 
transform. 
For applications of Fourier-Mukai transforms,
it is important to know the conditions to preserve the stability of coherent sheaves, and 
many results are obtained by many people.
Among them,  Bayer and Macri's work \cite{BM} is notable:
For a Fourier-Mukai transform $\Phi$,
they proved that there is an (contravariant) equivalence $\Psi$ such that
%$\Psi$ acts trivially on the Mukai lattice and
$\Phi \circ \Psi$ preserves Gieseker semistability for a general member of the 
moduli spaces. 

We assume that $\Pic(X)={\Bbb Z}H$ and set $(H^2)=2n$.
In \cite{CNY} and \cite{aCM:K3},
we studied Fourier-Mukai transform $\Phi$ associated to the moduli space of stable sheaves, and gave 
a condition under which $\Phi$ itself preserves Gieseker semistability
for a general member of the moduli space.
To be more precise, let $v_0=r_0 e^{\frac{q}{p}H}$ $(r_0>0)$ be 
a primitive isotropic Mukai vector and
$X'=M_H(v_0)$ the moduli space of stable sheaves with the Mukai vector $v_0$.
Assume that there is a universal family
${\cal E} \in \Coh(X \times X')$, and
let $\Phi_{X \to X'}^{{\cal E}^{\vee}}:{\bf D}(X) \to {\bf D}(X')$ be the Fourier-Mukai transform
whose kernel is the dual ${\cal E}^{\vee}$ of ${\cal E}$.
Let ${\cal M}_H(v)$ be the moduli stack of semistable sheaves $E$ with
the Mukai vector $v(E)=v$. 
If $pd-qr>0$,
$\langle v,e^{\frac{q}{p}H} \rangle >0$ and $n \geq p^2$, we
showed that $\Phi_{X \to X'}^{{\cal E}^{\vee}[1]}(E)$ is stable for a general 
$E \in {\cal M}_H(v)$,
where $v=(r,dH,a)$ and $\langle \;\;,\;\; \rangle$ is the
bilinear form on $H^*(X,{\Bbb Q})$. 
Although our condition $n \geq p^2$ is a very strong constraints on $X$ to apply, but 
it is a quite simple condition (compare the condition in Remark \ref{rem:r_0=1} (2)). 
In this paper, we shall treat the case where $\langle v,e^{\frac{q}{p}H} \rangle \leq 0$
and get a similar simple condition.

\begin{NB}
For a primitive isptropic Mukai vector $v_0=r_0 e^{\frac{q}{p}H}$ ($r_0>0$),
let $X'=M_H(v_0)$ be a fine moduli space  
and ${\cal E}$ a universal family. 
Let $\Phi_{X \to X'}^{{\cal E}}:{\bf D}(X) \to {\bf D}(X')$ be the Fourier-Mukai transform
whose kernel is ${\cal E}$.

Assume that $v_0$ is the Mukai vector of an ideal sheaf $I_x$ of a point $x \in X$.
Then we studied in \cite{CNY} that 
$\Phi_{X \to X'}^{\cal E}$ for the study of sheaf cohomology groups of a
general stable sheaves.
In particular we proved that $\Phi_{X \to X'}^{\cal E}(E)^{\vee}$ 
is Gieseker semistable for a general member $E$ of ${\cal M}_H(r,dH,a)$ if $d>0$ and $n \geq 2r$.
  
In this paper, we shall treat an arbitrary $r_0$.
\end{NB}

\begin{thm}[{Theorem \ref{thm:n}}]\label{thm:n0}
We set $X'=M_H(r_0 e^{-\frac{q}{p}H})$ and assume that there is a universal family
${\cal E}$.
Let $v=(r,dH,a)$ be a Mukai vector such that
\begin{equation*}
r \geq 0, \; pd-rq>0,-\langle e^{\frac{q}{p}H}, v \rangle \geq 0,\;
\langle v^2 \rangle \geq -2.
\end{equation*}
Assume that $n \geq (2r+p)p$.
Then
%\begin{enumerate}
%\item[(1)] 
$D_{X'} \circ \Phi_{X \to X'}^{{\cal E}}(E) \in {\cal M}_{H'}(v')$
for a general $E \in {\cal M}_H(v) \setminus Z$, where 
$D_{X'}:{\bf D}(X') \to {\bf D}(X')$ is the taking dual functor (see \eqref{eq:dualize}), 
$v'=\Phi_{X \to X'}^{{\cal E}}(v)$ and $Z$ is a closed substack of ${\cal M}_H(v)$. 
%\item[(2)]
%Assume that $a \geq 2$. 
%Then for a general $E \in M_H(v)$,
%we have the following claims.
%\begin{enumerate}
%\item
%$H^0(X,E) \otimes {\cal O}_X \to E$ is surjective,
%\item
%$H^1(X,E)=H^2(X,E)=0$
%for a general $E \in M_H(v)$
%\end{enumerate}
%\end{enumerate}
Moreover
$\dim Z \leq \dim {\cal M}_H(v)-2$ if $r>0$. 
\end{thm}
 
\begin{rem}\label{rem:r_0=1}
\begin{enumerate}
\item[(1)]
In \cite{CNY}, we treated the case where $r_0=1$.
%In \cite{aCM:K3}, we also studied $\Phi_{X \to X'}^{{\cal E}}(E)[1]$ for the case where 
%$\langle v, v_0 \rangle>0$ 
%and proved that Gieseker semistability is preserved for a general member of ${\cal M}_H(v)$ 
%if $n \geq p^2$.
\item[(2)]
In \cite{Y:Stability}, we proved that
there is a constant $N$ depending on $r$ and $\langle v^2 \rangle$ such that
$D_{X'} \circ \Phi_{X \to X'}^{{\cal E}}(E) \in {\cal M}_{H'}(v')$ 
for all $E \in {\cal M}_H(v)$ if $d \geq N$.  
\begin{NB}
Since $\Pic(X)={\Bbb Z}H$, $H$ is a general polarization.
\end{NB}
%\item[(3)]
%We note that $M_H(r_0 e^{\frac{p}{q}H}) \cong X'$ and there is a 
%universal family ${\cal E}_1$.
%%% be a universal family of stable sheaves on $X$ with the Mukai vector $r_0 e^{\frac{p}{q}H}$. 
%If ${\cal E}$ consists of a family of $\mu$-stable vector bundles, then
%we may choose ${\cal E}_1$ as ${\cal E}^{\vee}$ and 
%$\Phi_{X \to X'}^{{\cal E}}=\Phi_{X \to X'}^{{\cal E}_1^{\vee}}$. Thus we can 
%%%In our main result, we use $\Phi_{X \to X'}^{{\cal E}}$ which is Fourier-Mukai transform.
\end{enumerate}
\end{rem}

To get our result, we use the same strategy as in \cite{CNY}, \cite{aCM:K3}.
Thus we shall study totally semistable walls in the space of Bridgeland stability conditions.

\begin{rem}\label{rem:disprepant}
For an abelian surface, walls for Bridgeland stability corresponds to walls for birational geometry. 
On the other hand, for a K3 surface
we have discrepancy 
between walls for Bridgeland stability and walls for birational geometry. 
Indeed totally semistable walls by spherical objects do not affect birational geometry of the moduli space
(cf. sebsection \ref{subsect:n=1}).
\end{rem}

\begin{NB}
\begin{prob}
Assume that $\Phi_{X \to X'}^{{\cal E}}$ induces a birational map
of the moduli of stable sheaves.
Study the properties of ${\cal E}$.

For example, does ${\cal E}$ parameterize stable sheaves?
\end{prob}
\end{NB}

Let us explain the organization of this paper.
In section \ref{sect:stability1},
we recall Bridgeland stability conditions and walls for a K3 surface.
In particular we recall the parameter space of geometric stability conditions
for a K3 surface with the Picard number 1.
In section \ref{sect:wall}, we shall study totally semistable walls and prove
our main result.
In subsection \ref{subsect:tss-wall},
we describe totally semistable walls and give some conditions that there is no totally semistable walls
on a curve in the apace of stability conditions.
In subsection \ref{subsect:FM-preserve}, we study the action of Fourier-Mukai transform
on the space of geometric stability conditions, and prove our main results.
In subsection \ref{subsect:aCM}, we shall slightly refine the birational correspondence
in \cite{aCM:K3}. Thus we give a condition that the indeterminacy locus is of codimension 2
(Theorem \ref{thm:a<0}).
We also give an asymptotic result with a different flavor (Theorem \ref{thm:a<<0}).
In section \ref{sect:example},
we give birational automorphisms of $M_H(v)$ by $D_X \circ \Phi_{X \to X}^{{\cal E}}$
(Proposition \ref{prop:bir-auto}).
In subsection \ref{subsect:n=1}, we treat the case where $X$ is a K3 surface with $n=1$. 
In this case, the assumption of Theorem \ref{thm:n0} does not hold. 
We give examples of birational automorphisms of $M_H(v)$ by a Fourier-Mukai transform
whose kernel is not a family of stable sheaves. 
%In Appendix, we shall discuss the case of an abelian surface.
%Since there is no spherical object, we get a better condition to preserve
%the semistability.

{\it Notation.}
\begin{enumerate}
\item
Let ${\bf D}(X)$ be the bounded derived category of the category $\Coh(X)$ 
of coherent sheaves on $X$.
\item
Let $(H^*(X,{\Bbb Z})_{\alg},\langle\;\;,\;\; \rangle)$ be the 
algebraic part of the Mukai lattice of $X$.
Thus 
$$
H^*(X,{\Bbb Z})_{\alg}:={\Bbb Z} \oplus \NS(X) \oplus {\Bbb Z}
$$
 and
\begin{equation}
\langle x,y \rangle:=(x_1 \cdot y_1)-x_0 y_2-y_0 x_2 \in {\Bbb Z}
\end{equation}
for $x=(x_0,x_1,x_2), y=(y_0,y_1,y_2) \in  {\Bbb Z} \oplus \NS(X) \oplus {\Bbb Z}$.
\item
For an object $E \in {\bf D}(X)$, 
let $v(E):=\ch(E)\sqrt{\td_X} \in H^*(X,{\Bbb Z})_{\alg}$ be the 
Mukai vector of $E$. 
\item
For $v=(r,dH,a) \in H^*(X,{\Bbb Z})_{\alg}$,
${\cal M}_H(v)$ denotes the moduli stack of semistable sheaves $E$ with $v(E)=v$.
Let $M_H(v)$ be the moduli scheme of stable sheaves $E$ with $v(E)=v$.
It is an irreducible 
symplectic manifold of dimension $\langle v^2 \rangle+2$, if
$v$ is primitive. 
\item
For $D \in \NS(X)$,
\begin{equation}
e^D (x_0,x_1,x_2):=(x_0,x_1+x_0 D,x_2+(D \cdot x_1)+\tfrac{(D^2)}{2}x_0).
\end{equation} 
\item
For $E \in {\bf D}(X)$, we set 
\begin{equation}\label{eq:dualize}
D_X(E):={\bf R}{\cal H}om_{{\cal O}_X}(E,{\cal O}_X) \in {\bf D}(X).
\end{equation}
Then $D_X$ defines a contravariant functor ${\bf D}(X) \to {\bf D}(X)$. 
\end{enumerate}

\section{Stability conditions and walls.}\label{sect:stability1}

\subsection{Stability conditions}
Let us briefly recall some notation on stability conditions of Bridgeland
for a K3 surface $X$.
For more details, see \cite{Br:stability}, \cite{Br:3} and \cite{Br:4}.
A stability condition $\sigma=(Z_\sigma,{\cal P}_\sigma)$ 
on ${\bf D}(X)$
consists of a group homomorphism $Z_\sigma: {\bf D}(X) \to {\Bbb C}$
and a slicing ${\cal P}_\sigma$ of ${\bf D}(X)$ such that if
$0 \ne E \in {\cal P}_\sigma(\phi)$ then 
$Z_\sigma(E) =m(E)\exp(\pi \sqrt{-1} \phi)$ 
for some $m(E) \in {\Bbb R}_{>0}$.
The set of stability conditions has a structure of complex 
manifold.
We denote this space by $\Stab(X)$.

\begin{defn}\label{defn:B-stable}
\begin{enumerate}
\item[(1)]
A $\sigma$-semi-stable object $E$ of phase $\phi$ is
an object of  ${\cal P}_\sigma(\phi)$.
If $E$ is a simple object of ${\cal P}_\sigma(\phi)$, then
$E$ is $\sigma$-stable. 
\item[(2)]
For $v \in H^*(X,{\Bbb Z})_{\alg}$,
we denote the moduli stack of $\sigma$-semi-stable objects $E$ with 
$v(E)=v$ by ${\cal M}_\sigma(v)$,
where we usually choose 
$\phi:=\mathrm{Im}(\log Z_\sigma(v))/\pi \in (-1,1]$.
\end{enumerate}
\end{defn}

By \cite[Prop. 5.3]{Br:stability},
giving a stability condition $\sigma$ is
the same as giving a bounded $t$-structure on ${\bf D}(X)$ 
and a stability function $Z_\sigma$ on its heart ${\cal A}_\sigma$
with the Harder-Narasimhan property.
For $\sigma=(Z_\sigma,{\cal P}_\sigma)$,
we have the relation ${\cal A}_\sigma ={\cal P}_\sigma((0,1])$,
where ${\cal P}_\sigma((0,1])$ is the subcategory of ${\bf D}(X)$
generated by semi-stable objects $E \in {\cal P}_\sigma(\phi)$ with 
$\phi \in (0,1]$.  
Since the pair $(Z_\sigma,{\cal A}_\sigma)$
defines a stability condition, 
we also use the symbol $\sigma=(Z_\sigma,{\cal A}_\sigma)$
to denote a stability condition.
For a $\sigma$-semi-stable object $E \in 
{\cal P}_\sigma(\phi)$, we set $\phi_\sigma(E)=\phi$.
For $E \in {\cal P}_\sigma((0,1])$,
we also set $\phi_\sigma(E) \in (0,1]$ 
by $Z_\sigma(E)=m(E)\exp(\pi \sqrt{-1} \phi_\sigma(E))$. 

Let $\Stab(X)$ be the space of stability conditions.
For an equivalence $\Phi:{\bf D}(X) \to {\bf D}(X')$,
we have an isomorphism $\Phi:\Stab(X) \to \Stab(X')$
such that 
$\Phi(\sigma)$ $(\sigma \in \Stab(X))$ 
is a stability condition given by
\begin{equation}\label{eq:FM-action}
\begin{split}
Z_{\Phi(\sigma)}=& Z_\sigma \circ \Phi^{-1}:{\bf D}(X') \to {\Bbb C},\\
{\cal P}_{\Phi(\sigma)}(\phi)=& \Phi({\cal P}_\sigma(\phi)).
\end{split}
\end{equation}
We also have an action of the universal covering 
$\widetilde{\GL}_2^+({\Bbb R})$ of $\GL_2^+({\Bbb R})$ on 
$\Stab(X)$.
Since ${\Bbb C}^{\times} \subset  \GL_2^+({\Bbb R})$,
we have an injective homomorphism
${\Bbb C} \to  \widetilde{\GL}_2^+({\Bbb R})$.
Thus 
we have an action of $\lambda \in {\Bbb C}$ on
$\Stab(X)$. 
For a stability condition $\sigma \in \Stab(X)$, 
$\lambda(\sigma)$ is %a stability condition such that 
given by
\begin{equation}\label{eq:C-action}
\begin{split}
Z_{\lambda(\sigma)}=&
\exp(-\pi  \sqrt{-1} \lambda)Z_\sigma\\
{\cal P}_{\lambda(\sigma)}(\phi)=&
{\cal P}_\sigma(\phi+\mathrm{Re}\lambda).
\end{split}
\end{equation}

\subsection{Geometric stability conditions}
Let $\Delta^+(X)$ be the set of  
$(-2)$-vectors $v_1 \in H^*(X,{\Bbb Z})_{\alg}$
with $\rk v_1 >0$.
For 
$$
(\beta, \omega) \in 
\left(\NS(X)_{\Bbb R} \times \Amp(X)_{\Bbb R} \setminus 
\bigcup_{v_1 \in \Delta^+(X)}\left\{(\beta,\omega) \left|
\langle e^{\beta+\sqrt{-1} \omega},v_1 \rangle 
\in {\Bbb R}_{\leq 0} \right. \right\}\right),
$$
%$(\beta,\omega) \in \NS(X)_{\Bbb R} \times \Amp(X)_{\Bbb R}$,
Bridgeland \cite{Br:3} constructed a stability condition
$\sigma_{(\beta,\omega)}=(Z_{(\beta,\omega)},{\cal A}_{(\beta,\omega)})$
which is characterized
by the stability of ${\Bbb C}_x$ ($x \in X$), where
$Z_{(\beta,\omega)}(\bullet)=
\langle e^{\beta+\sqrt{-1}\omega},v(\bullet) \rangle:
{\bf D}(X) \to {\Bbb C}$ is the stability function and
${\cal A}_{(\beta,\omega)}$ 
is an abelian category which is a tilting of $\Coh(X)$
by a torsion pair $({\cal T}_{(\beta,\omega)},{\cal F}_{(\beta,\omega)})$:
\begin{enumerate}
\item
${\cal T}_{(\beta,\omega)}$ is a full subcategory of $\Coh(X)$ generated by
torsion sheaves and $\mu$-stable sheaves $E$ with 
$(c_1(E)-\rk E \beta,\omega)>0$ and
\item
${\cal F}_{(\beta,\omega)}$ is a full subcategory of $\Coh(X)$ generated by
$\mu$-stable sheaves $E$ with 
$(c_1(E)-\rk E \beta,\omega) \leq 0$.
\end{enumerate}
This kind of stability conditions are called geometric.
To be more precise,
for a geometric stability condition
$\sigma=(Z_\sigma,{\cal A}_\sigma)$,
we require that
$\mathrm{Re}\mho, \mathrm{Im}\mho$ span a positive definite
2-plane of $H^*(X,{\Bbb R})$,
where $Z_\sigma(\bullet)=\langle \mho,\bullet \rangle$
with $\mho \in H^*(X,{\Bbb Q})_{\alg} \otimes {\Bbb C}$. 
Up to the action of $\widetilde{\GL}_2^+({\Bbb R})$, 
there is $(\beta,\omega)$ such that
$\sigma=\sigma_{(\beta,\omega)}$.
We set ${\cal M}_{(\beta,\omega)}(v):=
{\cal M}_{\sigma_{(\beta,\omega)}}(v)$.

\subsection{Walls and cambers}\label{subsect:wall/chamber}

Let $\Stabd(X)$ denote the connected component of $\Stab(X)$ 
containing geometric stability conditions. 
For a fixed Mukai vector $v \in H^{*}(X,{\Bbb Z})_{\alg}$, there exists a locally finite set of \emph{walls} (real codimension one submanifolds with boundary) in $\Stabd(X)$, depending only on $v$, with the following properties (see \cite{Br:3}):
\begin{enumerate}
\item[(1)]
When $\sigma$ varies in a chamber, that is, a connected component of 
the complement of the union of walls, the sets of 
$\sigma$-semistable and $\sigma$-stable objects of class $v$ do not change. 
If $v$ is primitive, then $\sigma$-stability coincides with $\sigma$-semistability 
for $\sigma$ in a chamber for $v$.

\item[(2)]
 When $\sigma$ lies on a wall $W \subset\Stabd(X)$, 
there is a $\sigma$-semistable object of class $v$ that is unstable in one of the adjacent chambers 
and semistable in the other adjacent chamber.
 If $\sigma= (Z_\sigma, {\cal A}_\sigma)$ lies on a wall, there exists a $\sigma$-semistable object $E$ 
of Mukai vector $v$ and a subobject $F \subset E$ in ${\cal A}_\sigma$ 
with the same $\sigma$-slope but 
$v(F) \not \in {\Bbb R} v$.

\item[(3)]
 Given a polarization $H\in\Amp(X)$ and the Mukai vector $v$ of an $H$-Gieseker semistable sheaf, 
there exists a chamber ${\cal C}$ for $v$, the \emph{Gieseker chamber}, 
where the set of $\sigma$-semistable objects of class 
$v$ coincides with the set of $H$-Gieseker semistable sheaves \cite[Prop. 14.2]{Br:3}.
\end{enumerate}

%\section{Stability conditions on a K3 surface of Picard rank 1.}\label{sect:stability2}

%\subsection{Geometric stability conditions and walls.}

\subsection{The case of Picard number 1.}\label{subsect:rho=1}

Assume that 
$\Pic(X) = \mathbb{Z} H$, where $H$ is the ample generator with $(H^2)=2n$. 
Then writing $\beta = sH$ and $\omega = tH$, the stability conditions $\sigma_{(sH,tH)}$ 
correspond to points $(s,t)$ in the subset ${\Bbb H}^0$ of the upper half plane ${\Bbb H}$
defined by
\begin{equation}\label{eq:geometric-def}
{\Bbb H}^0={\Bbb H} \setminus \bigcup_{v_1 \in \Delta^+(X)} 
\left\{(s,t) \left| 
\langle e^{(s+t \sqrt{-1})H},v_1 \rangle \in {\Bbb R}_{\leq 0}
\right. \right\},
\end{equation}
where $\Delta^+(X)$ is the set of Mukai vectors $v_1=(r_1,d_1 H,a_1)$ 
such that $\langle v_1^2 \rangle=-2$ and $r_1>0$.

\begin{defn}\label{defn:W}
We fix a Mukai vector $v=(r,dH,a)$.
For a Mukai vector $v_1=(r_1,d_1 H,a_1)$ with $v_1 \not \in {\Bbb Q}v$,
let $W_{v_1}$ be a closed subset of ${\Bbb H}$ such that
$$
W_{v_1}=\{(s,t) \in {\Bbb H} \mid {\Bbb R} Z_{\sigma_{(sH,tH)}}(v)=
 {\Bbb R}Z_{\sigma_{(sH,tH)}}(v_1) \}.
$$
\end{defn}

For a wall $W$ in ${\Bbb H}^0$,
there is $v_1$ such that $W \subset W_{v_1}$.  

\begin{lem}[{\cite[Lem. 3.2]{aCM:K3}}]\label{lem:wall}
Let $W_{v_1}'$ be a connected component of $W_{v_1} \cap {\Bbb H}^0$.
\begin{enumerate}
\item[(1)]
$W_{v_1}$ is a semi-circle if $rd_1-r_1 d \ne 0$ and a vertical half line defined by $rs-d=0$ if $rd_1-r_1 d=0$.
\item[(2)]
For a point $(s_0,t_0) \in W_{v_1}'$,
we assume that there are $\sigma_{(s_0 H,t_0 H)}$-semistable 
objects $E_i$ $(i=1,2)$ such that $v(E_i)=v_i$ and
$\phi_{\sigma_{(s_0 H,t_0 H)}}(E_1)=\phi_{\sigma_{(s_0 H,t_0 H)}}(E_2)$.
Then $E_i$ are $\sigma_{(s H,t H)}$-semistable 
with $\phi_{\sigma_{(s H,t H)}}(E_1)=\phi_{\sigma_{(s H,t H)}}(E_2)$
for all $(s,t) \in W_{v_1}'$.
%\item[(2)]
%$W_{v_1} \cap U_\pm$ is connected.
\item[(3)]
$W_{v_1}'$ is a wall in ${\Bbb H}^0$ 
if and only if $W_{v_1}'$ is a wall at a point.   
\end{enumerate}
\end{lem}

\begin{rem}
Lemma \ref{lem:wall} (2) means that $\sigma_{(sH,tH)}$-semistability 
(resp. $\sigma_{(sH,tH)}$-stability) is independent 
along $W_{v_1}'$.
\end{rem}

\section{Walls and Fourier-Mukai transforms}\label{sect:wall}

\subsection{Totally semistable walls in $U_+^\beta$.}\label{subsect:tss-wall}

For $\beta:=\frac{q}{p}H$ $(p,q \in {\Bbb Z}, p \ne 0)$, let
\begin{equation}\label{eq:u0}
u=r_0 e^\beta,\;(r_0>0)
\end{equation}
 be a primitive isotropic Mukai vector.
For our study of Fourier-Mukai transforms in subsection \ref{subsect:FM-preserve},
we need a description of 
%our complicated parameter space 
${\Bbb H}^0$ 
%of geometric stubility conditions
in a neighborhood of the vertical line $ps-q=0$. 
We set
\begin{equation}
U_\pm^\beta:=\left\{(s,t) \left|
0< \pm \left(s-\frac{q}{p} \right)<\min \left\{ \frac{1}{r_0}\sqrt{\frac{1}{4n}},
\frac{1}{2r_0^2} \right\}, 
t>\frac{1}{r_0}\sqrt{\frac{1}{4n}}-\sqrt{\frac{1}{4r_0^2 n}-
\left(s-\frac{q}{p}\right)^2}
\right. \right\}.
\end{equation}
Then we have the following description.
\begin{prop}[{\cite[Prop. 2.11, Prop. 2.16]{FM-duality}}]\label{prop:stability-cond}
$$
U_\pm^\beta \subset {\Bbb H}^0.
$$
\end{prop}

\begin{rem}
Assume that $\gcd(p,q)=1$.
Let $U$ be a small neighborhood of $(\frac{q}{p},0) \in {\Bbb R}^2$. 
If $p \nmid (q^2 n+1)$, then
${\Bbb H}^0$ contains the half straight line defined by $ps-q=0$.
In this case, $U_+^\beta \cap U$ and $U_-^\beta \cap U$ belong to the same chamber
for $v=(0,0,1)$. 
On the other hand, if $p \mid (q^2 n+1)$, then
for the stability conditions $\sigma_+$ and $\sigma_-$
parameterized by $U_+^\beta \cap U$ and $U_-^\beta \cap U$ respectively,
$\sigma_+$ is far from $\sigma_-$ in the space of stability conditions. 
%belong to different sheet
%in the space of stability conditions.
\end{rem}

\begin{defn}
For $v=(r,dH,a)$,
we set 
\begin{equation}
r^\beta=r,\;d^\beta=d-r\frac{q}{p},\; a^\beta=a-2nd\frac{q}{p}+rn\frac{q^2}{p^2}.
\end{equation}
\end{defn}

Then
$v=e^\beta (r^\beta,d^\beta H,a^\beta)$ and
$pd^\beta,p^2 a^\beta \in {\Bbb Z}$.
For $v_1=(r_1,d_1 H,a_1)$, we have

\begin{equation}
\begin{split}
& d_1^\beta r-d^\beta r_1 \in {\Bbb Z},\quad
d_1^\beta a^\beta-d^\beta a_1^\beta
%=& \left(d_1-\frac{q}{p}r_1 \right) \left(a-2nd\frac{q}{p}+rn\frac{q^2}{p^2} \right)
%-\left(d-\frac{q}{p}r \right) \left(a_1-2nd_1 \frac{q}{p}+r_1n\frac{q^2}{p^2} \right)
\in \frac{1}{p^2}{\Bbb Z}.
\end{split}
\end{equation}
For a small $\nu>0$, we set
$$
U_\pm^\beta(\nu):=\left\{(s,t) \in U_\pm^\beta \mid 
0< \pm \left(s-\tfrac{q}{p} \right)<\nu \right\}.
$$

\begin{defn}
For a Mukai vector $v=(r,d H,a)$ with 
$r \geq 0$, $d^\beta>0$ and $a^\beta \geq 0$,
let $D_v^\beta$ be the set of Mukai vectors
$v_1=(r_1,d_1 H,a_1)=e^\beta (r_1,d_1^\beta H,a_1^\beta)$ such that
%$v_1 \ne v({\cal O}_X)$,
\begin{equation}
\langle v_1^2 \rangle=-2\epsilon\; (\epsilon=0,1),\;\;
d^\beta \geq d_1^\beta > 0,\;\;
\frac{a_1^\beta d^\beta-a^\beta d_1^\beta}{r_1 d^\beta-r d_1^\beta}>0
\end{equation}
and $r_1>r$ if $d_1^\beta=d^\beta$.
\end{defn}

\begin{NB}
Assume that $d_1^\beta=0$.
Since $d_1-r_1 s>0$ for $s \in U_+^\beta$, $r_1<0$.
Then we also get $a_1<0$.
So if $a^\beta \geq 0$, then $\langle v_1,v \rangle=-r_1 a^\beta-r a_1^\beta>0$ unless $r=0$. 
\end{NB}

\begin{NB}
\begin{rem}\label{rem:symmetry}
For $v_1=(r_1,d_1 H,a_1) \in D_v^\beta$,
$d_1^\beta=d^\beta$ implies $\frac{a_1^\beta-a^\beta}{r_1-r}>0$.
Hence we have a bijection
\begin{NB2}
\begin{equation}
\begin{matrix}
D_{e^\beta(r,d^\beta H,a^\beta)}^\beta & \to & 
D_{e^{\beta'}(r_0 a^{\beta'},d^{\beta'} H,r/r_0)}^{\beta'}\\
e^\beta(r_1,d_1^\beta H,a_1^\beta) & \mapsto & e^{\beta'}(r_0 a_1^{\beta'},d_1^{\beta'} H,r_1/r_0).
\end{matrix}
\end{equation}
\end{NB2}
\begin{equation}
\begin{matrix}
D_v^\beta & \to & 
D_{v'}^{\beta'}\\
v_1 & \mapsto & v_1',
\end{matrix}
\end{equation}
where
$v':=\Phi_{X \to X'}^{\cal E}(v)^{\vee}, v_1':=\Phi_{X \to X'}^{\cal E}(v_1)^{\vee}$.
\end{rem}
\end{NB}

\begin{prop}[{\cite[Prop. 3.6]{CNY}}]\label{prop:birational}
Let $v=(r,dH,a)$ be a Mukai vector such that
$r \geq 0$, $d^\beta,a^\beta>0$ and $\langle v^2 \rangle \geq -2$. 
Then the following conditions are equivalent.
\begin{enumerate}
\item[(1)]
${\cal M}_H(v) \cap {\cal M}_{(sH,tH)}(v) \ne \emptyset$
for $(s,t) \in U_+^\beta(\nu)$ with $(0<\nu \ll 1)$.
\item[(2)]
\begin{enumerate}
\item
If $\langle v^2 \rangle>0$, then
$\langle v,v_1 \rangle \geq 0$ for $v_1 \in D_v^\beta$ with $\epsilon=1$ and
$\langle v,v_1 \rangle \geq 2$ for $v_1 \in D_v^\beta$ with $\epsilon=0$.
\item
If $\langle v^2 \rangle \leq 0$, then
$\langle v,v_1 \rangle \geq 0$ for $v_1 \in D_v^\beta$ with $\epsilon=1$. 
\end{enumerate}
\end{enumerate}
\end{prop}

\begin{proof}
The proof is similar to that of \cite[Prop. 3.6]{CNY}.
\end{proof}

\begin{rem}
Assume that $d_1=d$. If $r_1 d-rd_1>0$, then $r_1>rd_1/d=r$.
\begin{NB}
In this case, $C_\lambda$ passes $(0,t_0)$.
\end{NB}
\end{rem}

The following result is well-known.
\begin{lem}\label{lem:stability}
If  $t_1 \leq t \leq t_2$ and $(s,t_1) \in {\Bbb H}^0$, then 
$$
{\cal M}_{(sH,t_1 H)}(v) \cap {\cal M}_{(sH,t_2 H)}(v) \subset {\cal M}_{(sH,t H)}(v).
$$
In particular if ${\cal M}_{(sH,t_1 H)}(v) \cap {\cal M}_{(sH,t_2 H)}(v)  
\ne \emptyset$, then
${\cal M}_{(sH,t_1 H)}(v) \cap {\cal M}_{(sH,t H)}(v) \ne \emptyset$ 
for all $t_1 \leq t \leq t_2$.  
\end{lem}

\begin{NB}
\begin{proof}
For $E \in M_{(sH,t_1 H)}(v) \cap M_{(sH,t_2 H)}(v)$,
let $E_1$ be a subobject of $E$ in ${\cal A}_{(sH,tH)}$.
Since ${\cal A}_{(sH,tH)}$ is independent of the choice of $t$,
$E_1$ is a subobject of $E$ in ${\cal A}_{(sH,t'H)}$ for all 
$t_1 \leq t' \leq t_2$.
Then $\phi_{(sH,t_i H)}(E_1) \leq \phi_{(sH,t_i H)}(E)$ ($i=1,2$)
imply that $\phi_{(sH,tH)}(E_1) \leq \phi_{(sH,tH)}(E)$.
\end{proof}
\end{NB}

\begin{lem}\label{lem:r_1<0}
Let $v=(r,dH,a)$ be a Mukai vector such that
$r,a^\beta \geq 0$, $d^\beta>0$ and $\langle v^2 \rangle \geq -2$.
Assume that $v_1=(r_1,d_1 H,a_1) \in D_v^\beta$ satisfies
$r_1 d^\beta-r d_1^\beta<0$.
\begin{enumerate}
\item[(1)]
Assume that $n \geq p^2$ and $\langle v^2 \rangle \geq 0$. 
Then $\langle v,v_1 \rangle \geq 2$ and the inequality is strict
unless $n=p^2$ and $r=a^\beta=0$.
\item[(2)]
Assume that $n \geq p(r+p)$ and $\langle v^2 \rangle = -2$. 
Then $\langle v,v_1 \rangle \geq 2$.
\end{enumerate}.

\end{lem}

\begin{proof}
We note that $a_1^\beta d^\beta-a^\beta d_1^\beta<0$.
\begin{NB}
$v$ may not be primitive.
\end{NB}
If $r_1 \leq 0$, then $n (d_1^\beta)^2+\epsilon=r_1 a_1^\beta$ and $r_1 \ne 0$ imply
$r_1,a_1^\beta<0$. 
Hence
$$
\langle v,v_1 \rangle = 2nd^\beta d_1^\beta-r_1 a^\beta-r a_1^\beta 
\geq 2+a^\beta+r/p^2 \geq 2.
$$
\begin{NB}
The claim holds if $(d_1^\beta)^2 n \geq 1$.
If $(d_1^\beta)^2 n < 1$, then $d_1^\beta <1$.
For example, if $n=2$ and $p=2$, then $d_1^\beta=1/2$ and $(d_1^\beta)^2 n=1/2$.
\end{NB}
If the equality holds, then 
$n=p^2$ and $r=a^\beta=0$.
%If the equality holds, then $n=d=d_1=1$, $a+r=0$. Hence
%$v=(0,H,0)$. If $\langle v_1^2 \rangle=0$, then 
%we also have $v_1=(-1,H,-1)$.
\begin{NB}
If $d^\beta > p/2$, then
$d^\beta d_1^\beta >1/2$. Hence
$\langle v,v_1 \rangle>1$.
\end{NB}
 
If $r_1 >0$, then
$n (d_1^\beta)^2+\epsilon=r_1 a_1^\beta<r_1 a^\beta \frac{d_1^\beta}{d^\beta}$.
We first assume that $\langle v^2 \rangle \geq 0$.
By $r_1 d^\beta-r d_1^\beta \leq -1$,
we get $r>0$ and
$$
\frac{d_1^\beta}{r_1}n \geq \frac{d^\beta}{r} n +\frac{n}{r r_1}
\geq \frac{a^\beta}{d^\beta}+\frac{n}{r r_1}>\frac{a^\beta}{d^\beta}.
$$
Hence 
$$
r_1 d_1^\beta \frac{d_1^\beta}{r_1}n>r_1 a^\beta \frac{d_1^\beta}{d^\beta}>
n (d_1^\beta)^2+\epsilon \geq n(d_1^\beta)^2,
$$
which is a contradiction. Therefore (1) holds.
\begin{NB}
If $r_1>0$, then the assumption $n \geq p^2$ is not used. 
\end{NB}

Assume that $\langle v^2 \rangle=-2$. 
Then we also get $r>0$ and
$$
\frac{d_1^\beta}{r_1}n 
\geq \frac{a^\beta}{d^\beta}+\frac{1}{r} \left(\frac{n}{r_1}-\frac{1}{d^\beta} \right) =
\frac{a^\beta}{d^\beta}+\frac{1}{r} \left(\frac{n}{r_1}-\frac{p}{pd-qr} \right).
$$
Since $d_1^\beta \leq d^\beta$, we get $r_1<\frac{d^\beta_1}{d^\beta}r \leq r$.
Hence $n \geq pr \geq pr_1$.
Therefore 
$\frac{d_1^\beta}{r_1}n  \geq \frac{a^\beta}{d^\beta}$.
If $\epsilon=1$, then 
$$
r_1 d_1^\beta \frac{d_1^\beta}{r_1}n \geq
n (d_1^\beta)^2+1 > n(d_1^\beta)^2,
$$
which is a contradiction.
Therefore (2) holds.

\end{proof}

\begin{lem}\label{lem:r_1>0}
Let $v=(r,dH,a)$ be a Mukai vector such that
$r,a^\beta \geq 0$, $d^\beta>0$.
Assume that $v_1=(r_1,d_1 H,a_1) \in D_v^\beta$ satisfies
$r_1 d^\beta-r d_1^\beta>0$.
We set $k:=a_1^\beta d^\beta-a^\beta d_1^\beta>0$.
Then the following claims hold:
\begin{enumerate}
\item[(1)]
\begin{equation}\label{eq:r_1>0}
\langle v,v_1 \rangle = \frac{1}{r_1 d_1}
\left(
(r_1 d^\beta-r d_1^\beta)(n(d_1^\beta)^2-\epsilon)-(2rd_1^\beta \epsilon -k r_1^2) \right).
\end{equation}
\item[(2)]
If $\langle v_1^2 \rangle=0$ and $n > p^2$, then
$\langle v,v_1 \rangle \geq 3$.
\item[(3)]
Assume that $p=1$ and $n=1$.
If $\langle v_1^2 \rangle=0$, then
$\langle v,v_1 \rangle \geq 3$ unless $v=(r,(r+1)H,r)$ and $v_1=(1,H,1)$.  
\end{enumerate}
\end{lem}

\begin{proof}
(1)
We note that $r_1>0$. Then we see that 
\begin{equation}
\begin{split}
\langle v,v_1 \rangle =& 2nd^\beta d_1^\beta-r_1 a^\beta-r a_1^\beta \\
= & 2nd^\beta d_1^\beta-(a_1^\beta d^\beta-k)\frac{r_1}{d_1^\beta}-ra_1^\beta\\
=&\frac{d^\beta}{d_1^\beta}(n(d_1^\beta)^2-r_1 a_1^\beta)-ra_1^\beta +\frac{r_1 k}{d_1^\beta}
+nd^\beta d_1^\beta \\
=&-\frac{d^\beta}{d_1^\beta}\epsilon-\frac{r}{r_1}((d_1^\beta)^2 n+\epsilon)+
\frac{r_1 k}{d_1^\beta}+nd^\beta d_1^\beta \\
=& \frac{d^\beta}{d_1^\beta}(n(d_1^\beta)^2-\epsilon)-
\frac{1}{r_1 d_1^\beta}(rn (d_1^\beta)^3+rd_1^\beta \epsilon-k r_1^2)\\
%=&(nd_1^2-\epsilon)\frac{1}{d_1 r_1}
%\left(r_1 d-r d_1-\frac{2r d_1 \epsilon-r_1^2}{nd_1^2-\epsilon}\right).
=& \frac{1}{r_1 d_1^\beta}\left(
(r_1 d^\beta-r d_1^\beta)(n(d^\beta_1)^2-\epsilon)-(2rd_1^\beta \epsilon -k r_1^2) \right).
\end{split}
\end{equation}

(2)
Note that $k \geq 1/p^2$.
If $\epsilon=0$ and $n>p^2$, then
\begin{equation}\label{eq:e=0}
\begin{split}
\langle v,v_1 \rangle =  &
(r_1 d^\beta-r d_1^\beta)
n\frac{d_1^\beta}{r_1}+k\frac{r_1}{d_1^\beta}
\geq 2\sqrt{kn(r_1 d^\beta-r d_1^\beta)} > 2.
%\left(r_1 d-r d_1+\frac{r_1^2}{nd_1^2}\right) \\
%\geq &
%\frac{nd_1^2}{d_1 r_1}2 \sqrt{(r_1 d-r d_1)\frac{r_1^2}{nd_1^2}} \\
%= & 2\sqrt{n(r_1 d-r d_1)} \geq 2.
\end{split}
\end{equation}

(3)
If $\langle v,v_1 \rangle=2$, then by the proof of (2), we see that
$r_1 d-rd_1=1$, $k=a_1 d-ad_1=1$ and $d_1=r_1$.
Then we have $r_1=d_1=1$, which implies $a_1=1$.
Hence we get $d=r+1=a+1$. Therefore $v=(r,(r+1)H,r)$ and $v_1=(1,H,1)$.  
\end{proof}

\begin{NB}
If $p \mid n$, then $pk \in {\Bbb Z}$. 
\end{NB}

\begin{lem}\label{lem:n>2p^2}
Let $v_1=(r_1,d_1 H,a_1) \in D_v^\beta$ be a Mukai vector satisfying
$r_1 d^\beta-r d^\beta_1>0$ and $\langle v_1^2 \rangle=-2$. 
Assume that $n \geq 2p^2$ and $\langle v,v_1 \rangle<0$.
Then $r_1<r$ and $d_1^\beta<\frac{2r}{p^2}$.
\end{lem}

\begin{proof}
By $n \geq 2p^2$ and $pd_1^\beta \in {\Bbb Z}$,
$(n-p^2)(d_1^\beta)^2\geq (p d_1^\beta)^2 \geq 1$.
Hence 
\begin{equation}
\begin{split}
& (r_1 d-r d_1)(n (d_1^\beta)^2-1)-(2rd_1^\beta-kr_1^2)\\
\geq & n(d_1^\beta)^2-1-2r d_1^\beta+k r_1^2 \\
\geq & (pd_1^\beta)^2-2r d_1^\beta+k r_1^2 \\
=&
\left(pd_1^\beta-\tfrac{r}{p} \right)^2-\frac{r^2}{p^2}+k r_1^2
\end{split}
\end{equation}
If $r_1 \geq r$, then $kr_1^2 \geq \frac{r^2}{p^2}$.
If $d_1^\beta \geq \frac{2r}{p^2}$, then
$(pd_1^\beta)^2-2r d_1^\beta=d_1^\beta(p^2 d_1^\beta-2r) \geq 0$.
Therefore the claim holds.
\end{proof}

\subsection{Preservation of Gieseker stability.}\label{subsect:FM-preserve}

For $r_0$ in \eqref{eq:u0}, let
$X':=M_H(r_0 e^{-\beta})$ be the moduli space of stable sheaves
$E$ with $v(E)=r_0 e^{-\beta}$.
Assume that $X'$ is a K3 surface (i.e., projective) and there is a
a coherent sheaf ${\cal E}$ on $X \times X'$ which
is a universal family of stable sheaves.
Let us consider the Fourier-Mukai
transform $\Phi_{X \to X'}^{{\cal E}}:{\bf D}(X) \to {\bf D}(X')$ defined by
$$
\Phi_{X \to X'}^{{\cal E}}(E):=
{\bf R}p_{X' *}(p_X^*(E) \otimes {\cal E}),
$$
where $p_X:X \times X' \to X$ and $p_{X'}:X \times X' \to X'$ are projections.
We shall study the condition under which $D_{X'} \circ \Phi_{X \to X'}^{{\cal E}}$
preserves Gieseker stabilility for a general member of ${\cal M}_H(v)$.
%By the action of  on ${\Bbb H}$,
%$U_-^\beta$ is maped to $U_+^{\beta'}$,
By \cite[Prop. 2.13, Prop. 2.16]{FM-duality}, we have the following result.
\begin{lem}\label{lem:stability-cond}
If $(s,t) \in U_+^\beta (\nu)$, then
${\cal E}_{| X \times \{x' \}}^{\vee}$ is $\sigma_{(sH,tH)}$-stable for all $x' \in X'$.  
\end{lem}

We set 
$$
s_0=\frac{q}{p},\quad
\beta' =c_1({\cal E}_{|\{ x \} \times X'}^{\vee})/r_0,\quad
\beta'=s_0' H',
$$
 where $\NS(X')={\Bbb Z}H'$.
Then
$$
v({\cal E}_{|\{ x \} \times X'})=r_0 e^{-\beta'}.
$$
For $(s,t) \in U_+^\beta$, we define
$(s',t') \in {\Bbb H}$ by
\begin{equation}
s':=s_0'+\frac{(s-s_0)}{nr_0((s-s_0)^2+t^2)},\;
t':=\frac{t}{nr_0((s-s_0)^2+t^2)}
\end{equation}
(see \cite[2.1]{FM-duality}). 
Then Proposition \ref{prop:stability-cond} and Lemma \ref{lem:stability-cond} imply that
$$
\sigma_{(-s'H',t'H')} \equiv \Phi_{X \to X'}^{{\cal E}}(\sigma_{(sH,tH)})
$$
modulo the action of $\widetilde{GL}_2^+({\Bbb R})$ on $\Stabd(X)$ (see \cite[Prop. 2.2]{FM-duality}).
By $D_{X'} \circ \Phi_{X \to X'}^{{\cal E}}$, we have an isomorphism
$$
{\cal M}_{(sH,tH)}(v) \simto {\cal M}_{(s'H',t'H')}(v').
$$
Let ${\cal C}_0^\beta \subset U_+^\beta$ be a chamber for $v$
whose closure contains $(s_0,0)$
and ${\cal C}_0^{\beta'} \subset U_+^{\beta'}$ be a chamber for $v'$
whose closure contains $(s_0',0)$.
We note that 
\begin{equation}
\begin{split}
t \gg 0 & \Longleftrightarrow (s',t' ) \in {\cal C}_0^{\beta'},\\
t' \gg 0 & \Longleftrightarrow (s,t) \in {\cal C}_0^\beta.
\end{split}
\end{equation}
%$t$ is sufficiently large implies $(s',t' ) \in {\cal C}_0^{\beta'}$
%and
%$(s,t) \in {\cal C}_0^\beta$ implies $t'$ is sufficiently large.
Thus we have isomorphisms
\begin{equation}
\begin{split}
& {\cal M}_H(v) 
%\to 
\simto  {\cal M}_{(s'H',t'H')}(v'),\quad   ((s',t') \in {\cal C}_0^{\beta'})  \\
& {\cal M}_{(sH,tH)}(v) 
\simto
{\cal M}_{H'}(v'),\quad   ((s,t) \in {\cal C}_0^\beta),
\end{split}
\end{equation}
where $v=(r,dH,a)$ with $d^\beta>0$ and $r,a^\beta \geq 0$.
\begin{NB}
Since $v=e^\beta(r,d^\beta H,a^\beta)$ and $v'=e^{\beta'}(r_0 a^\beta,d^\beta H',r/r_0)$,
${\cal M}_{(s'H',t'H')}(v')$ consists of Gieseker stable sheaves if 
$(s,t) \in {\cal C}_0^\beta$.
\end{NB}
To be more precise, let us describe ${\cal C}_0^\beta$ in detail. 
We note that $E, {\cal E}_{|X \times \{x' \} }^{\vee}[1] \in {\cal A}_{(sH,tH)}$ for
$s_0<s<\frac{d}{r}$, where $E \in {\cal M}_{(sH,tH)}(v)$.
Let $C$ be a semi-circle defined by
\begin{equation}
t^2+(s-s_0)\left(s-s_0-\frac{2a^\beta}{d^\beta (H^2)} \right)=0.
\end{equation}
Then $\phi_{(sH,tH)}({\cal E}_{|X \times \{x' \} }^{\vee}[1])=\phi_{(sH,tH)}(E)$
for $(s,t) \in C$
(see the proof of \cite[Prop. 5.6, Prop. 5.9]{YY2}).
\begin{NB}
The expression also works for $r=0$.
If $\langle v^2 \rangle=-2$, then
$C$ passes $(\frac{d^\beta}{r},\sqrt{\frac{2}{r_1^2 (H^2)}})$.
\end{NB}
If $(s,t)$ satisfies 
$$
t>\sqrt{(s-s_0)\left(\frac{2a^\beta}{d^\beta (H^2)}+s_0-s \right)},\;
1 \gg s-s_0>0,
$$
 then
$\phi_{(sH,tH)}({\cal E}_{| X \times \{x' \}}^{\vee}[1])>\phi_{(sH,tH)}(E)$.
Let ${\cal C}$ be a chamber whose closure contains
$C$ and $\phi_{(sH,tH)}({\cal E}_{|X \times \{x' \}}^{\vee}[1])>\phi_{(sH,tH)}(E)$ for
$(s,t) \in {\cal C}$.
Then ${\cal C}={\cal C}_0^\beta$.
Thus we get the following.

\begin{NB}
$C$ is transformed to the virtical line $s'-s_0'=\frac{d^\beta}{r_0 a^\beta}$.
\end{NB}

\begin{prop}[{\cite[Thm. 4.9]{MYY:2011:2}}]\label{prop:isom}
Assume that $(s,t) \in {\cal C}$. Then
we have an isomorphism
$$
{\cal M}_{(sH,tH)}(v) \cong {\cal M}_{H'}(v')
$$
by sending $E \in {\cal M}_{(sH,tH)}(v)$ to
$\Phi_{X \to X'}^{{\cal E}}(E)^{\vee}$, where
$v'=\Phi_{X \to X'}^{\cal E}(v)^{\vee}$.
\end{prop}

\begin{prop}\label{prop:no-wall}
The following conditions are equivalent. 
\begin{enumerate}
\item[(1)]
$\Phi_{X \to X'}^{{\cal E}}(E)^{\vee} \in {\cal M}_{H'}(v')$
for a general $E \in {\cal M}_H(r,dH,a)$.
\item[(2)]
$\langle v,v_1 \rangle \geq 0$
for $v_1=(r_1,d_1 H,a_1) \in D_v^\beta$ with $\langle v_1^2 \rangle=-2$
and $r_1 d-r d_1>0$.  
\end{enumerate}
Under the condition, we have a birational map
\begin{equation}
\begin{matrix}
{\cal M}_H(r,dH,a) & \cdots \to & {\cal M}_{H'}(v')\\
E & \mapsto & \Phi_{X \to X'}^{{\cal E}}(E)^{\vee} 
\end{matrix}
\end{equation}
\end{prop}

\begin{proof}
By Proposition \ref{prop:stability-cond}, there is a segment 
$I(x) \in {\Bbb H}^0 (0 \leq x \leq 1)$ 
such that $I(0)=(0,t_\infty)$ is in the Gieseker chamber and
$I(1)=(s_1,t_1) \in {\cal C}$ and $0<s_1 \ll 1$. 
\begin{NB}
Let $(s(x) H,t(x) H)$ $(s(x) \geq 0)$ be a path in the space of 
geometric stability 
conditions connecting 
Gieseker chamber to ${\cal C}$.
\end{NB}
By Proposition \ref{prop:birational}, Lemma \ref{lem:r_1<0} and Lemma \ref{lem:r_1>0},
${\cal M}_H(r,dH,a) \cap {\cal M}_{(s_1 H,t_1 H)}(v) \ne \emptyset$
if and only if (2) holds.
Combining Proposition \ref{prop:isom},
we see that (1) is equivalent to (2).
\end{proof}

\begin{thm}\label{thm:n}
Let $v=(r,dH,a)$ be a Mukai vector such that
\begin{equation}\label{eq:cond-v}
r \geq 0, \; pd-rq>0,-\langle e^{\frac{q}{p}H}, v \rangle \geq 0,\;
\langle v^2 \rangle \geq -2.
\end{equation}
Assume that $n \geq (2r+p)p$.
Then
%\begin{enumerate}
%\item[(1)] 
$\Phi_{X \to X'}^{{\cal E}}(E)^{\vee} \in {\cal M}_{H'}(v')$
for a general $E \in {\cal M}_H(v) \setminus Z$, where 
$v'=\Phi_{X \to X'}^{{\cal E}}(v)$ and $Z$ is a closed substack of ${\cal M}_H(v)$. 
%\item[(2)]
%Assume that $a \geq 2$. 
%Then for a general $E \in M_H(v)$,
%we have the following claims.
%\begin{enumerate}
%\item
%$H^0(X,E) \otimes {\cal O}_X \to E$ is surjective,
%\item
%$H^1(X,E)=H^2(X,E)=0$
%for a general $E \in M_H(v)$
%\end{enumerate}
%\end{enumerate}
Moreover
$\dim Z \leq \dim {\cal M}_H(v)-2$ if $r>0$. 
\end{thm}

\begin{proof}
We note that $n>p^2$.
By Lemma \ref{lem:r_1<0} and Lemma \ref{lem:r_1>0},
it is sufficient to prove that
$\langle v,v_1 \rangle >0$ for $v_1=(r_1,d_1 H,a_1) \in D_v^\beta$
with $r_1>0$ and $\langle v_1^2 \rangle=-2$.
By $n \geq (2r+p)p$ and
\eqref{eq:r_1>0}, we see that
\begin{equation}
\begin{split}
& n(d_1^\beta)^2-1-2r d_1^\beta+k r_1^2\\
\geq & (2r+p)p (d_1^\beta)^2-2r d_1^\beta-1+kr_1^2 \\
=& 2r(pd_1^\beta-1)d_1^\beta+(pd_1^\beta)^2-1+kr_1^2>0.\\
\end{split}
\end{equation}
Hence the condition in Proposition \ref{prop:no-wall} (2) holds.
Therefore our claim holds.
\end{proof}

We also give a different condition for the birational correspondence.

\begin{prop}
Let $v=(r,dH,a)$ be a Mukai vector satisfying \eqref{eq:cond-v}.
Assume that $n \geq 2p^2$ and $d-r\frac{q}{p} \geq 2r+2\frac{r^2}{p^2}$.
Then 
$\Phi_{X \to X'}^{{\cal E}}(E)^{\vee} \in {\cal M}_{H'}(v')$
for a general $E \in {\cal M}_H(v) \setminus Z$, where 
$v'=\Phi_{X \to X'}^{{\cal E}}(v)$ and $Z$ is a closed substack 
with $\dim Z \leq \dim {\cal M}_H(v)-2$.
%$D_{X'} \circ \Phi_{X \to X'}^{{\cal E}}$ induces a birational map
%${\cal M}_H(r,dH,a) \to  {\cal M}_H(v')$, where 
%$v'=\Phi_{X \to X'}^{{\cal E}}(v)$. 
\end{prop}

\begin{proof}
Assume that there is a totally semistable wall.
By Lemma \ref{lem:r_1<0} and Lemma \ref{lem:r_1>0},
it is defined by a Mukai verctor $v_1=(r_1,d_1 H,a_1) \in D_v^\beta$ satisfying
$r_1 d^\beta-r d^\beta_1>0$ and $\langle v_1^2 \rangle=-2$. 
We may assume that 
$r_1<r$ and $d_1^\beta<\frac{2r}{p^2}$ (Lemma \ref{lem:n>2p^2}).
Then
$r_1 d-r d_1 \geq d^\beta-rd_1^\beta \geq 2r$.
Hence
\begin{equation}
\begin{split}
& (r_1 d-r d_1)(n (d_1^\beta)^2-1)-(2rd_1^\beta-kr_1^2)\\
\geq & 2r(n (d_1^\beta)^2-1-d_1^\beta)+kr_1^2 \\
\geq & 2r((p^2 (d_1^\beta)^2-1)+d_1^\beta (p^2 d_1^\beta-1))+kr_1^2>0.
\end{split}
\end{equation} 
Hence the condition in Proposition \ref{prop:no-wall} (2) holds, and hence
our claim holds.
\end{proof}

\subsection{A remark on \cite{aCM:K3}.}\label{subsect:aCM}

As a remark on \cite[Thm. 4.2]{aCM:K3}, we get the following refinement.

\begin{lem}\label{lem:aCM:remark}
Let $v=(r,dH,a)$ be a Mukai vector such that
$r \geq 0,a^\beta \leq 0$, $d^\beta>0$.
\begin{enumerate}
\item
[(1)]
If $n \geq p^2$ and $p \nmid q^2 n+1$, then
$\langle v,v_1 \rangle>0$ for $v_1 \in D_v^\beta$.
\item
[(2)]
If $n>p^2$, then $\langle v,v_1 \rangle \geq 3$ if $\epsilon=0$.
\end{enumerate}

\end{lem}

\begin{proof}
We set $k:=a_1^\beta d^\beta-a^\beta d_1^\beta$ as in Lemma \ref{lem:r_1>0}.

(1)
If $r_1 d^\beta-r d_1^\beta>0$, then we see that
\begin{equation}\label{eq:aCM:remark1}
\langle v,v_1 \rangle=\frac{2 d_1^\beta n}{r_1}(r_1 d^\beta-r d_1^\beta)+
\frac{r}{r_1}((d_1^\beta)^2 n-\epsilon)-r_1 a^\beta>0.
\end{equation}

If $r_1 d^\beta-r d_1^\beta<0$, then 
we get $k<0$, which implies $a_1<\frac{d_1^\beta}{d^\beta}a^\beta \leq 0$.
Hence we also get
$r_1<0$. 
By using \eqref{eq:r_1>0}, we get
\begin{equation}\label{eq:aCM:remark2}
\langle v,v_1 \rangle=
\frac{1}{-r_1 d_1^\beta}\left(
-(r_1 d^\beta-r d_1^\beta)(n(d^\beta_1)^2-\epsilon)+(2rd_1^\beta \epsilon -k r_1^2) \right)>0.
\end{equation}

(2)
We first assume that 
$r_1 d^\beta-r d_1^\beta<0$.
%We set $k_1:=-(a_1^\beta d^\beta-a^\beta d_1^\beta)>0$.
Then $k<0$, $a_1^\beta<\frac{d_1^\beta}{d^\beta} a^\beta \leq 0$ and
$-k \geq 1/p^2$.
If $\epsilon=0$ and $n>p^2$, then
\begin{equation}
\begin{split}
\langle v,v_1 \rangle =  &
(r d_1^\beta-r_1 d^\beta)
n\frac{d_1^\beta}{-r_1}-k \frac{-r_1}{d_1^\beta}
\geq 2\sqrt{(-k) n(r d_1^\beta-r_1 d^\beta)} > 2.
%\left(r_1 d-r d_1+\frac{r_1^2}{nd_1^2}\right) \\
%\geq &
%\frac{nd_1^2}{d_1 r_1}2 \sqrt{(r_1 d-r d_1)\frac{r_1^2}{nd_1^2}} \\
%= & 2\sqrt{n(r_1 d-r d_1)} \geq 2.
\end{split}
\end{equation}

We next assume that $r_1 d^\beta-r d_1^\beta>0$.
%We set $k_2:=a_1^\beta d^\beta-a^\beta d_1^\beta>0$.
Note that $r_1>r \frac{d_1^\beta}{d^\beta} \geq 0$ and
$k \geq 1/p^2$.
If $\epsilon=0$ and $n>p^2$, then
\begin{equation}
\begin{split}
\langle v,v_1 \rangle =  &
(r_1 d^\beta-r d_1^\beta)
n\frac{d_1^\beta}{r_1}+k \frac{r_1}{d_1^\beta}
\geq 2\sqrt{k n(r_1 d^\beta-r d_1^\beta)} > 2.
%\left(r_1 d-r d_1+\frac{r_1^2}{nd_1^2}\right) \\
%\geq &
%\frac{nd_1^2}{d_1 r_1}2 \sqrt{(r_1 d-r d_1)\frac{r_1^2}{nd_1^2}} \\
%= & 2\sqrt{n(r_1 d-r d_1)} \geq 2.
\end{split}
\end{equation}

\end{proof}

Let ${\cal E} \in {\bf D}(X \times X')$ be a complex such that
${\cal E}^{\vee}$ is the universal family of stable sheaves with the Mukai vector
$r_0 e^\beta$, where $X' \cong M_H(r_0 e^\beta)$.
Let us consider Fourier-Mukai transform
$\Phi_{X \to X'}^{{\cal E}}:{\bf D}(X) \to {\bf D}(X')$.
If $p \nmid q^2 n+1$, then 
${\cal E}$ is a family of $\mu$-stable locally free sheaves on $X$.
By Lemma \ref{lem:aCM:remark}, 
we can estimate the dimension of the indeterminacy locus of the birational correspondence.

\begin{thm}\label{thm:a<0}
Let $v=(r,dH,a)$ be a Mukai vector such that
$$
r \geq 0, \; pd-rq>0,\;
\langle v^2 \rangle \geq 0.
$$ 
Assume that $n > p^2$.
% and $p \nmid q^2 n+1$.
If $\langle v,v_0 \rangle \geq 0$, then
$\Phi_{X \to X'}^{{\cal E}}(E)[1] \in {\cal M}_{H'}(v')$
for a general $E \in {\cal M}_H(v) \setminus Z$, where  
$v'=-\Phi_{X \to X'}^{{\cal E}}(v)$ and $Z$ is a closed substack of ${\cal M}_H(v)$
such that $\dim Z \leq \dim {\cal M}_H(v)-2$.
\end{thm}

By modifying the proof of Theorem \ref{thm:a<0},
we can also show the following claim.
 
\begin{thm}\label{thm:a<<0}
For a fixed $r$ and $d$
with
$$
r>0,\; pd-qr>0,
$$
if 
$$
a^\beta \leq \min \left\{-\frac{p(pd-qr)}{r}, -p, -r,-2(pd-qr) \right\},
$$ 
then
$\Phi_{X \to X'}^{{\cal E}}(E)[1] \in {\cal M}_{H'}(v')$
for a general $E \in {\cal M}_H(v) \setminus Z$, where  
$v'=-\Phi_{X \to X'}^{{\cal E}}(v)$ and $Z$ is a closed substack of ${\cal M}_H(v)$
such that $\dim Z \leq \dim {\cal M}_H(v)-2$.
%$\Phi_{X \to X'}^{{\cal E}[1]}$ induces a birational map.
\end{thm}

\begin{proof}
(1)
Assume that $\langle v,v_1 \rangle \leq 2$ for $v_1 \in D_v^\beta$ with 
$\langle v_1^2 \rangle=0$.
By the proof of Lemma \ref{lem:aCM:remark} (2), we get
\begin{equation}\label{eq:nkd}
nk(r_1 d^\beta-rd_1^\beta) \leq 1.
\end{equation}
We shall prove that
\begin{equation}\label{eq:aCM:remark3}
a^\beta > \min \left\{-\frac{p(pd-qr)}{r}, -p \right\}.
\end{equation}

We first assume that $r_1 d^\beta-rd_1^\beta<0$ and
$k=a_1^\beta d^\beta-a^\beta d_1^\beta<0$.
Since $-kp^2 \in {\Bbb Z}_{>0}$, \eqref{eq:nkd} implies
$n(r_1 d^\beta-rd_1^\beta) \geq -p^2$.
By $r_1<0$, we have $d_1^\beta \leq \frac{p^2}{rn}$.
Since $r_1 a_1^\beta=(d_1^\beta)^2 n$,
$$
a^\beta>\frac{d^\beta}{d_1^\beta}a_1^\beta=
\frac{d^\beta d_1^\beta}{r_1}n>\frac{d^\beta p^2}{r_1 r}>\frac{d^\beta p^2}{-r}.
$$

We next assume that $r_1 d^\beta-rd_1^\beta>0$
and
$k=a_1^\beta d^\beta-a^\beta d_1^\beta>0$.
Then $kn \leq 1$. Hence
$$
a^\beta>\frac{a_1^\beta d^\beta-k}{d_1^\beta} \geq 
\frac{a_1^\beta d^\beta-1}{d_1^\beta}>-p.
$$

(2)
Assume that $\langle v,v_1 \rangle \leq 0$ for $v_1 \in D_v^\beta$ with 
$\langle v_1^2 \rangle=-2$.
We shall prove that
\begin{equation}\label{eq:aCM:remark4}
a^\beta > \min \left\{-r, -2(pd-qr) \right\}.
\end{equation}
If $r_1 d^\beta-rd_1^\beta>0$,
then by using the computation in \eqref{eq:aCM:remark1}, we see that  
$$
0 \geq \langle v,v_1 \rangle \geq -\frac{r}{r_1}-r_1 a^\beta.
$$
Hence $a^\beta>-r$.
If $r_1 d^\beta-rd_1^\beta<0$,
then the computation in \eqref{eq:aCM:remark2} implies that
$n(d_1^\beta)^2<1$.
Hence $-a_1^\beta=\frac{(d_1^\beta)^2n+1}{(-r_1)}<2$.
Therefore 
$$
a^\beta \geq \frac{d^\beta}{d_1^\beta}a_1^\beta > -2(pd-qr).
$$

By (1) and (2), our claim holds.
\end{proof}

\section{Examples}\label{sect:example}

The birational automorphims of moduli spaces
are extensively studied in \cite{BC} and explicit description of some examples
are done in \cite{FMM}.
In this section, 
let us give some examples of birational automorphisms
by using Theorem \ref{thm:n}.
We first recall the action of Fourier-Mukai transforms
on the Mukai lattice in \cite[sect. 2]{Mori-Y}.
We set
\begin{equation}\label{eq:A}
A=\begin{pmatrix}
q \sqrt{n} & c\\
-p & -q \sqrt{n}
\end{pmatrix},
\end{equation}
where $c,p,q \in {\Bbb Z}$ and
$\det A=q^2 n-pc=\pm 1$.
Then
$A$ defined an isometry $\phi_A$ of $H^*(X,{\Bbb Z})_{\alg}$ by
$\phi_A((x,yH,z))=(x',y' H,z')$:
\begin{equation}\label{eq:phi}
\begin{split}
\begin{pmatrix}
x' & y' \sqrt{n}\\
y' \sqrt{n} & z'
\end{pmatrix}=&
{^t} A \begin{pmatrix}
x & y \sqrt{n}\\
y \sqrt{n} & z
\end{pmatrix}
A\\
=&
\begin{pmatrix}
q^2 nx-2pqn y+p^2 z & cq x-(pc+q^2 n) y+pq z \\
cq  x-(pc+q^2 n) y+pq z & c^2 x-2cqn  y+q^2 nz
\end{pmatrix}.
\end{split}
\end{equation}
Since $A^2=-(\det A) I_2$, $\phi_A$ is an involution, where $I_2$ is the
identity matrix.

\begin{NB}
\begin{equation}
\begin{split}
&
\begin{pmatrix}
q\sqrt{n} & -p\\
c & -q \sqrt{n}
\end{pmatrix}
\begin{pmatrix}
r & d \sqrt{n}\\
d \sqrt{n} & a
\end{pmatrix}
\begin{pmatrix}
q\sqrt{n} & c \\
-p & -q \sqrt{n}
\end{pmatrix}\\
=& 
\begin{pmatrix}
q\sqrt{n} & -p\\
c & -q \sqrt{n}
\end{pmatrix}
\begin{pmatrix}
(rq-dp)\sqrt{n} & rc-dqn \\
dqn-ap & \left(d c-aq\right) \sqrt{n}
\end{pmatrix}\\
=& 
\begin{pmatrix}
q^2 nr-2dpqn+ap^2 & \left(rc-2dq^2 n+d+apq\right)\sqrt{n} \\
\left(rc-2dq^2 n+d+apq\right)\sqrt{n} & rc^2-2cqnd+aq^2 n 
\end{pmatrix}
\end{split}
\end{equation}
\end{NB}

\subsection{The case where $\det A=-1$.}\label{subsect:A=-1}
Assume that $\det A=-1$. 
We note that $M_H(p^2,-pq H,q^2 n) \cong X$.
Let ${\cal E} \in \Coh(X \times X)$ be the universal family of stable sheaves
with the Mukai vector $(p^2,-pqH,q^2 n)$, that is, we have an identification
\begin{equation}\label{eq:dual}
\begin{matrix}
X & \stackrel{\sim}{\longrightarrow} & M_H(p^2,-pqH,q^2 n) \\
x & \mapsto & {\cal E}_{|X \times \{x\} }. 
\end{matrix}
\end{equation}
Let us consider Fourier-Mukai transform
\begin{equation}\label{eq:kernel-A}
\Phi_{X \to X}^{{\cal E}}:{\bf D}(X) \to {\bf D}(X).
\end{equation}
Since
$$
v(\Phi_{X \to X}^{{\cal E}}({\cal E}_{|X \times \{ x \}}^{\vee}))=(0,0,1)=\phi_A(v({\cal E}_{|X \times \{ x \}}^{\vee})),
$$
we can take the universal family ${\cal E}$ to satisfy 
$$
(D_X \circ \Phi_{X \to X}^{{\cal E}})_{|H^*(X,{\Bbb Z})_{\alg}}=\phi_A
$$
(see \cite{Mori-Y}).
\begin{NB}
$$
v({\cal E}_{|\{ x \} \times X}^{\vee})=v(D_X \circ \Phi_{X \to X}^{{\cal E}}({\Bbb C}_x))=(p^2,pq H,q^2 n).
$$
\end{NB}

\begin{rem}\label{rem:stable-kernel} 
Let $\Phi_{X \to X}^{{\cal E}'}:{\bf D}(X) \to {\bf D}(X)$ be 
another Fourier-Mukai transform
which induces the same action on the Mukai lattice as $\Phi_{X \to X}^{{\cal E}}$.
Then ${\cal E}'_{|X \times \{ x \}}$ are simple complexes for all $x \in X$.
If they are simple sheaves for all $x \in X$, then
they are stable by \cite{Mukai:Tata}, and hence
we get ${\cal E'} \cong {\cal E}$.
Therefore ${\cal E}'$ is not a family of simple sheaves.
\end{rem}

\begin{lem}[{\cite[Remark 2.15]{Mori-Y}}]\label{lem:T(X)}
Let $T(X)$ be the transcendental lattice of $X$.
Then the action of $D_X \circ \Phi_{X \to X}^{{\cal E}}$ on $T(X)$ is
trivial.
\end{lem}

\begin{proof}
Since
$-(pc+q^2 n)=-2q^2 n+1 \equiv 1 \mod 2n$, $\phi_A$ acts trivially on 
$$
H^*(X,{\Bbb Z})_{\alg}^{\vee}/H^*(X,{\Bbb Z})_{\alg}=\NS(X)^{\vee}/\NS(X) \cong
{\Bbb Z}/2n{\Bbb Z}.
$$
Since 
$$
T(X)^{\vee}/T(X) \cong \NS(X)^{\vee}/\NS(X)
$$
and the isometry group $O(T(X))$ consists of $\pm 1$ (\cite[Lem. 4.1]{Oguiso}), the claim holds.
\end{proof}

\begin{prop}\label{prop:bir-auto}
Let $v=(r,dH,a)$ $(r>0)$ be a Mukai vector such that 
$$
a=\frac{2npqd-(q^2 n-1)r}{p^2}.
$$
Then the following claims hold:
\begin{enumerate}
\item[(1)]
$\phi_A(v)=v$.
\item[(2)]
If $n \geq 2rp+p^2$, then
$D_X \circ \Phi_{X \to X}^{{\cal E}}$ induces a birational involution of $M_H(v)$ which is non-symplectic.
\end{enumerate}
\end{prop}

\begin{proof}
(1) follows from \eqref{eq:phi}.

(2)
By Theorem \ref{thm:n}, $D_X \circ \Phi_{X \to X}^{{\cal E}}$ induces a birational map
$M_H(v) \cdots \to M_H(v)$.
Under the identification 
$$
v^\perp \cong H^2(M_H(v),{\Bbb Z}),
$$
the induced action on $v^\perp$ is the homomorphism $-D_X \circ \Phi_{X \to X}^{{\cal E}}$.
Hence it is a non-symplectic action.
\end{proof}

\begin{rem}
$v=e^{\frac{q}{p}H}(r,(d-\tfrac{q}{p}r)H,\tfrac{r}{p^2})$.
Hence $\langle v^2 \rangle/2=(d-\tfrac{q}{p}r)^2 n-\frac{r^2}{p^2}$.
\end{rem}

\begin{rem}
For the existence of a birational involution, we do not need the assumption $n \geq 2rp+p^2$ by the work of
Bayer and Macri \cite{BM}.  Our assumption is used to specify the kernel of $\Phi_{X \to X}^{{\cal E}}$. 
\end{rem}

\begin{ex}\label{ex:involution}
Assume that $p^2 \mid(q^2 n-1)$. 
Let $r>0$ and $d$ be integers satisfying 
$p \mid 2d$. 
Then 
$$
a:=\frac{2npqd-(q^2 n-1)r}{p^2} \in {\Bbb Z}.
$$
For $v:=(r,dH,a)$, we can apply Proposition \ref{prop:bir-auto}.
In particular we have a non-symplectic birational automorphism
of $X^{[\ell+1]}$, where $\ell=\frac{q^2 n-1}{p^2}$.
\end{ex}

\begin{NB}
Assume that $\ell p^2-n q^2=-1$.
If $p=3$, then
$p^2 \mid q^2 n-1$ iff $q \equiv \pm 1 \mod 9$ and $n \equiv 1 \mod 9$ or
$q \equiv \pm 2 \mod 9$ and $n \equiv 7 \mod 9$ or
$q \equiv \pm 4 \mod 9$ and $n \equiv 4 \mod 9$.
If $n \geq 15$, then
$D_X \circ \Phi_{X \to X}^{{\cal E}}$ induces a
birational automorphism of $X^{[\ell+1]}$ which is non-symplectic, where
$\ell=\frac{q^2 n-1}{9}$.
\end{NB}

\begin{NB}
$n=2t$. $n \geq 8$.
We set
$X'=M_H(2,p_2 H,p_2^2 t)$ $(p_2>0)$, where
$p_2 q_1 t-2q_1=1$, that is
$q_2=\frac{p_2 q_1 t-1}{2}$.
Then there is a universal family ${\cal E}$ which is locally free.

$v=(p_2^2 t,p_2 q_2 H,2q_2^2)-\ell(2,q_1 H,q_1^2 t)$.
For a general $E \in {\cal M}_{H'}(v)$, we have an exact sequence
$$
0 \to E \to E_0 \to \oplus_{i=1}^{\ell+1} {\cal E}_{|\{ x_i \} \times X'} \to 0
$$
where
$E_0 \in {\cal M}_{H'}(p_2^2 t+2,(p_2 q_2+q_1)H',2q_2^2+q_1^2 t)$.

We set $v=-(p_2^2 t,p_2 q_2 H,2q_2^2)+\ell(2,q_1 H,q_1^2 t)$.
For a general $E \in {\cal M}_{H'}(v)$, we have an exact sequence
$$
0 \to E_0 \to \oplus_{i=1}^{\ell+1} {\cal E}_{|\{ x_i \} \times X'} \to E \to 0.
$$
Indeed $\Phi_{X \to X'}^{\cal E}=\Phi_{X \to X'}^{{\cal F}^{\vee}}$, where ${\cal F}={\cal E}^{\vee}$. 

\end{NB}

\subsection{The case where $\det A=1$.}
Assume that $\det A=1$.
Then $v_0=(p,qH,c)$ satisfies $\langle v_0^2 \rangle=2(q^2 n-pc)=-2$.

\begin{lem}\label{lem:reflection}
$$
\phi_A(u)=-(u+\langle u,v_0 \rangle v_0),\; u \in H^*(X,{\Bbb Z})_{\alg}.
$$
Thus $-\phi_A$ is the reflection by $v_0$.
\end{lem}

\begin{NB}
$\langle (x,yH,z),(p,qH,c) \rangle=2nqy-pz-cx$.
\begin{equation}
\begin{split}
u+\langle u,v_0 \rangle v_0=&
(x(1-cp)+2pqny-p^2 z,-(cqx-(1+2q^2 n)y+pqz)H,-c^2 x+2cqny+z(1-pc))\\
=& (-q^2 n x+2pqny-p^2 z,-(cqx-(pb+q^2 n)y+pqz)H,-c^2 x+2cqny-q^2 nz).
\end{split}
\end{equation}
\end{NB}

\begin{prop}
Assume that $p \mid q^2 n+1$ and $E_0$ be a stable rigid vector bundle 
$v(E_0)=(p,qH,\frac{q^2 n+1}{p})$.
Let $v=(r,dH,a)$ be a Mukai vector such that
$$
r \geq 0,\; pd-qr>0,\; -\langle v,v_0 \rangle>0
$$
 and $\langle v^2 \rangle \geq -2$. 
Assume that $n \geq (2r+p)p$.
Then
\begin{enumerate}
\item[(1)]
$\Ext^1(E_0,E)=0$ for a general $E \in {\cal M}_H(v)$.
\item[(2)]
$\coker(E^{\vee} \to \Hom(E_0,E)^{\vee} \otimes E_0^{\vee} )$
is semistable for a general $E \in {\cal M}_H(v)$. 
\end{enumerate}
\end{prop}

\begin{NB}
$$
{\cal E}:=\ker(E_0^{\vee} \boxtimes E_0 \to {\cal O}_\Delta)
$$
is the universal family.
If $n \geq 2rp+p^2$ and $-\langle e^{\frac{q}{p}H},v \rangle>0$, then
%$\Ext^1(E_0,E)=0$ for a general $E \in {\cal M}_H(r,dH,a)$.
$$
\Phi_{X \to X}^{{\cal E}}(E)^{\vee}=\mathrm{Cone}(E^{\vee} \to \Hom(E_0,E)^{\vee} \otimes E_0^{\vee})
$$ 
is a stable sheaf for a general
$E \in {\cal M}_H(v)$.
\end{NB}

\subsection{A special case.}

In the notation of \eqref{eq:A},
we assume that $p=2$ and $c:=\frac{q^2 n-1}{2} \in {\Bbb Z}$.
Then
$$
A=\begin{pmatrix}
 q\sqrt{n} & \frac{q^2 n-1}{2}\\
-2 & - q\sqrt{n}
\end{pmatrix}.
$$
is a matrix with $\det A=-1$.
In this case we shall explain that the universal family ${\cal E}$ of stable sheaves defining \eqref{eq:kernel-A} is
a non-locally free sheaf and $\Phi_{X \to X}^{{\cal E}}$ is related to a $(-2)$-reflection.
We set
$$
B=\begin{pmatrix}
 q\sqrt{n} & \frac{q^2 n+1}{2}\\
-2 & - q\sqrt{n}
\end{pmatrix}.
$$
Then $\det B=1$,
$v_0=(2,qH,\frac{q^2 n+1}{2})$ is a $(-2)$-vector
and $-\phi_B$ is the reflection by $v_0$ (Lemma \ref{lem:reflection}).
We see that
$$
B \begin{pmatrix}
1 & -q \sqrt{n}\\
0 & 1
\end{pmatrix}
\begin{pmatrix}
1 & 0 \\
0 & -1
\end{pmatrix}=A,
$$
that is,
$$
e^{-qH} \phi_B (u)=\phi_A(u)^{\vee},\; u \in H^*(X,{\Bbb Z})_{\alg}.
$$
Let $E_0$ be a stable sheaf with $v(E_0)=v_0$.
The we get
$$
{\cal E} \cong \ker(E_0^{\vee} \boxtimes E_0 (-qH) \to {\cal O}_\Delta (-qH)).
$$
In particular, the birational map $M_H(v) \cdots \to M_H(v)$ in Proposition \ref{prop:bir-auto} is given by
$$
E \mapsto \coker(E^{\vee}(qH) \to \Hom(E_0,E) \otimes E_0^{\vee}(qH))
$$
(cf. \cite{FMM}).

\begin{NB}

Assume that $p \mid (q^2 n-1)$.
$$
C:=\begin{pmatrix}
 q\sqrt{n} & \frac{q^2 n-1}{p}\\
-p & - q\sqrt{n}
\end{pmatrix}.
$$
Then $\det C=-1$.

\begin{equation}
\begin{split}
&
\begin{pmatrix}
q\sqrt{n} & -p\\
\frac{q^2 n-1}{p} & -q \sqrt{n}
\end{pmatrix}
\begin{pmatrix}
r & d \sqrt{n}\\
d \sqrt{n} & a
\end{pmatrix}
\begin{pmatrix}
q\sqrt{n} & \frac{q^2 n-1}{p} \\
-p & -q \sqrt{n}
\end{pmatrix}\\
=& 
\begin{pmatrix}
q\sqrt{n} & -p\\
\frac{q^2 n-1}{p} & -q \sqrt{n}
\end{pmatrix}
\begin{pmatrix}
(rq-dp)\sqrt{n} & r\frac{q^2 n-1}{p}-dqn \\
dqn-ap & \left(d \frac{q^2 n-1}{p}-aq\right) \sqrt{n}
\end{pmatrix}\\
=& 
\begin{pmatrix}
q^2 nr-2dpqn+ap^2 & \left(\frac{q(q^2 n-1)r}{p}-2dq^2 n+d+apq\right)\sqrt{n} \\
\left(\frac{q(q^2 n-1)r}{p}-2dq^2 n+d+apq\right)\sqrt{n} & 
\end{pmatrix}
\end{split}
\end{equation}

%Then the action of $D_X \circ \Phi_{X \to X}^{{\cal E}}$ is $\phi_C$.

\end{NB}

\begin{NB}

\begin{prop}
Assume that $n \geq 8$ and $n \equiv 1 \mod 4$.
Let $E_0$ be a stable vector bundle such that $v(E_0)=(2,qH,2\ell+1)$ with
$q<0$ and
$\ell=\frac{q^2 n-1}{4} \in {\Bbb Z}$.
Then 
we have a non-symplectic 
birational automorphism $\varphi$ of $X^{[\ell+1]}$:
\begin{equation}
\begin{matrix}
\varphi:& X^{[\ell+1]} & \to & X^{[\ell+1]}\\
& I_Z & \mapsto & \coker(I_Z^{\vee}(qH) \to E_0^{\vee}(qH)).
\end{matrix}
\end{equation} 
\begin{NB2}
$v(E_0)-v(I_Z)=(1,qH,3 \ell+1)=(1,0,-\ell)e^{qH}$.
\end{NB2}
\end{prop}

%If $p^2 \ell-q^2 n=1$ and $n \geq p^2$, then $\Phi_{X \to X}^{{\cal E}[1]}$ induces a birational 
%automprphism of $X^{[\ell+1]}$.

\end{NB}

\subsection{An Example on a K3 surface with $n=1$.}\label{subsect:n=1}

As an example for the case of $n < 2rp+p^2$,
we assume that $X$ is a K3 surface with $n=1$
and treat the case where $(p,q)=(1,0)$, that is $\beta=0$.
For the ideal sheaf ${\cal E}:=I_\Delta$ of the diagonal $\Delta \subset X \times X$,
we have an identification \eqref{eq:dual} and an autoequivalence $\Phi_{X \to X}^{{\cal E}}$ of ${\bf D}(X)$.

We take a Mukai vector $v=(r,dH,a)$ with $r,d,a>0$.
Then
for an isotropic vector $v_1=(r_1,d_1 H,a_1) \in D_v^\beta$,
Lemma \ref{lem:r_1<0} and Lemma \ref{lem:r_1>0} imply that
$\langle v,v_1 \rangle \geq 3$ unless $n=1$, $v=(r,(r+1)H,r)$ and $v_1=(1,H,1)$.
If $0<d-r<\frac{1}{2}r$, then 
$(1,H,2)$ and $(2,H,1)$ define totally semistable walls for $v=(r,dH,r)$.
Hence $D_X \circ \Phi_{X \to X}^{{\cal E}}$ does not induce a birational automorphism of $M_H(v)$.
By \cite{BM}, there is a (contravariant) equivalence 
$\Psi:{\bf D}(X) \to {\bf D}(X)$ such that $D_X \circ \Phi_{X \to X}^{{\cal E}} \circ \Psi$  
induces a birational automorphism of $M_H(v)$.

Assume that $d \geq r+2$. 
By the following lemma, $\Psi$ acts trivially on $H^*(X,{\Bbb Z})$ unless $v$ is a multiple
of $(2,3H,2)$.
In particular the induced action of the birational automorphism on $v^\perp \cong H^2(M_H(v),{\Bbb Z})$ is
$-D_X \circ \Phi_{X \to X}^{{\cal E}}$. 
Obviously the kernel of $\Phi_{X \to X}^{{\cal E}} \circ \Psi$ is not a family of
stable sheaves on $X$ (Remark \ref{rem:stable-kernel}).

\begin{lem}
Let $v=(r,dH,r)$ $(r,d>0)$ be a primitive Mukai vector. 
Assume that $\langle v_1,v \rangle=0$ for $v_1 \in D_v^\beta$ with
$\langle v_1^2 \rangle=-2$.
Then $v=(2,3H,2)$ and $v_1=(1,H,2),(2,H,1)$.
\end{lem}

\begin{proof}
We note that $\gcd(r,d)=1$.
We assume that $\langle v_1,v \rangle=0$ for $v_1=(r_1,d_1 H,a_1) \in D_v^\beta$ with
$\langle v_1^2 \rangle=-2$.
Then 
\begin{equation}\label{eq:vv_1}
d_1^2+1=r_1 a_1,\;
2dd_1=r(a_1+r_1).
\end{equation}
By Lemma \ref{lem:r_1<0}, we get
\begin{equation}\label{eq:r_1a_1}
r_1>\frac{r d_1}{d},\; a_1 >\frac{r d_1}{d}.
\end{equation}

(1) We first treat the case where $r$ is odd.
Then we get 
\begin{equation}
d_1=kr,\; r_1+a_1=2kd\;\; (k \in {\Bbb Z}).
\end{equation}
%By \eqref{eq:r_1a_1},
%$r_1 \geq \frac{kr^2+1}{d}$ and $a_1\geq \frac{kr^2+1}{d}$.
If $r_1 \geq kd+1$, then by using \eqref{eq:r_1a_1}, we see that
$$
r_1 a_1\geq (kd+1)a_1=kda_1+a_1>krd_1+a_1=d_1^2+a_1
%(kd+1)\frac{kr^2+1}{d}=kr^2+k+\frac{kr^2+1}{d} >d_1^2+1.
$$
which contradicts \eqref{eq:vv_1}. 
Hence $r_1 \leq kd$. We also have $a_1 \leq kd$.
By $r_1+a_1=2kd$, $r_1=a_1=kd$.
Then $(kd)^2=(kr)^2+1$, and hence $r=0$ and $kd=1$, which is a contradiction.

(2) We next assume that $r$ is even.
Since $\gcd(r,d)=1$, $d$ is odd.
We set $r=2r'$.
Then 
\begin{equation}
d_1=kr',\; r_1+a_1=kd\; (k \in {\Bbb Z}). 
\end{equation}
By \eqref{eq:r_1a_1},
$r_1 \geq \frac{krr'+1}{d}$ and $a_1\geq \frac{krr'+1}{d}$.
Assume that $k$ is odd.
If $r_1 \geq \frac{kd+3}{2}$, then by using \eqref{eq:r_1a_1}, we see that
\begin{equation}
r_1 a_1\geq \left(\frac{kd+1}{2}+1 \right)a_1=
\left(\frac{kd+1}{2}\right)a_1+a_1>\frac{krd_1}{2}+a_1=d_1^2+a_1,
%\geq \left(\frac{kd+1}{2}\right)\frac{krr'+1}{d}+a_1>
%(kr')^2+1.
\end{equation}
which contradicts \eqref{eq:vv_1}. 
Hence $r_1 \leq \frac{kd+1}{2}$. We also have $a_1 \leq \frac{kd+1}{2}$.
By $r_1+a_1=kd$, we see that
$\{r_1,a_1 \}=\{ \frac{kd-1}{2},\frac{kd+1}{2}\}$.
Hence we get 
$$
d_1^2+1=r_1 a_1=\frac{(kd)^2-1}{4}.
$$
%Since $r_1 a_1 \equiv 1 \mod 4$, we get a cotradiction. 
Then we see that $kd=3$ and $d_1=1$, and hence,
$k=r'=1$ and $d=3$. Therefore 
$v=(2,3H,2)$ and $v_1=(1,H,2),(2,H,1)$.

Assume that $k$ is even.
If $r_1 \geq \frac{kd+2}{2}$, then we see that
\begin{equation}
r_1 a_1\geq \left(\frac{kd}{2}+1 \right)a_1=
\frac{kd}{2}a_1+a_1>\frac{krd_1}{2}+a_1=d_1^2+a_1,
\end{equation}
which contradicts \eqref{eq:vv_1}. 
Hence $r_1 \leq \frac{kd}{2}$.
We also have $a_1 \leq \frac{kd}{2}$.
By $r_1+a_1= kd$, we see that $r_1=a_1=\frac{kd}{2}$. 
By solving the equation $\frac{(kd)^2}{4}=d_1^2+1$, 
we get $d_1=0$, which is a contradiction. 
\end{proof}

\begin{NB}
For a general $n$, we also get the following.
Assume that $n \nmid r$, $\gcd(r,d)=1$ and $r$ is odd.
Then a similar claim holds.

\begin{proof}
We set $l:=\gcd(n,r)$. Then
$r=r' l,n=n' l$ $(l \in {\Bbb Z})$.
We note that $r'$ is odd, $\gcd(r',dn')=1$, $n'>1$. 

We see that $d_1=r' k$, $r_1+a_1=2n' dk$.
If $r_1>\frac{ndd_1}{r}=n'dk$, then
$r_1 \geq n' dd_1+1$.
Then we get
$$
r_1 a_1 \geq n'dk a_1+a_1>nd_1^2+a_1. 
$$
Therefore $r_1 \leq n'dk$. similarly
we get $a_1 \leq n'dk$.
By $r_1+a_1=2n' dk$, $r_1=a_1=n'dk$.
Then we get $(n'dk)^2=d_1^2 n+1$, which contradicts $n'>1$.
Therefore there is no $v_1$.
\end{proof}

For $n=3$ and $v=(3,2H,3)$, $\langle v,(2,H,2) \rangle=0$
and define a codimension 1 wall.

\begin{lem}
Assume that $n > r$.
Then there is no totally semistable wall and codimension 1 wall
for $v=(r,dH,r)$. 
\end{lem}

\begin{proof}
Assume that $v_1 \in D_v^0$ with $\langle v,v_1 \rangle \leq 0$.
We note that
$$
\frac{ndd_1+1}{r}\frac{rd_1+1}{d}>nd_1^2+\frac{n}{r}d_1 \geq nd_1^2+1.
$$ 
Hence $r_1,a_1 \leq \frac{ndd_1}{r}$, which implies 
$\langle v,v_1 \rangle \geq 0$.
If the equality holds, then
$$
\left(\frac{ndd_1}{r} \right)^2=d_1^2 n+1.
$$
Then we get $\gcd(n,r)=n$, which contradicts to our assumption $n>r$.
Therefore $\langle v,v_1 \rangle>0$. 
\end{proof}
\end{NB}

In general it is not so easy to classify totally semistable walls and hence 
the autoequivalence $\Psi$.
As an easy example, we shall treat the case where $v=(r,(r+2)H,r)$ ($r$ is odd and $r \geq 5$).
We first classfy $v_1 \in D_v^\beta$ satisfying 
$\langle v,v_1 \rangle \leq 0$.

\begin{lem}
We set $v:=(r,(r+2)H,r)$ $(r \geq 5)$. 
Assume that $v_1=(r_1,d_1,a_1) \in D_v^\beta$ satisfies
$\langle v_1^2 \rangle=-2$ and $\langle v,v_1 \rangle \leq 0$.
Then $v_1=(1,H,2),(2,H,1)$.
\end{lem}

\begin{proof}
By Lemma \ref{lem:r_1<0},
$r_1 d-r d_1>0$ and $a_1 d-a d_1>0$.
Since $d_1 \leq d=r+2$,
$$
r_1>\frac{r}{r+2}d_1=d_1-\frac{2d_1}{r+2} \geq d_1-2.
$$
Moreover $r_1>d_1-1$ if $d_1<\frac{r}{2}+1$.
Hence 
$r_1 \geq r-1$ and $r_1 \geq r$ if $d_1<\frac{r}{2}+1.$
We also have $a_1  \geq r-1$ and $a_1 \geq r$ if $d_1<\frac{r}{2}+1$.

Assume that $d_1 \geq 4$.
If $r_1 \geq d_1+2$, then 
$$
r_1 a_1 \geq (d_1+2)(d_1-1)>d_1^2+1=r_1 a_1.
$$
Hence $r_1 \leq d_1+1$.
We also get $a_1 \leq d_1+1$.
Then we see that $r_1 a_1 \ne d_1^2+1$.
Assume that $d_1 \leq 3$.
Then $d_1<\frac{r}{2}+1$, which implies
$r_1,a_1 \geq d_1$.
Then we see that $r_1 a_1 \ne d_1^2+1$ unless $v_1=(1,H,2), (2,H,1)$. 
\end{proof}

We note that $F_1:={\cal O}_X(H)$ and stable vector bundle $F_2$ with 
$v(F_2)=(2,H,1)$ are stable on $U_+^\beta(\nu)$. 
Let $\Phi_1$ and $\Phi_2$ be spherical twists
associated to $F_1$ and $F_2$.
Then we can show that 
\begin{equation}\label{eq:Psi}
\Psi=\Phi_2 \circ \Phi_2 \circ \Phi_1 \circ \Phi_1
\end{equation}
%In particular $\Psi$ acts trivially on $H^*(X,{\Bbb Z})$.
by using the following lemma.

\begin{lem}\label{lem:wall-L}
For the Mukai vector $v=(r,dH,r)$ with $r>2(d-r)$,
we set $w_i:=v+\langle v,v(F_i) \rangle v(F_i)$.
Let $v_1=(r_1,d_1 H,a_1)$ $(d_1>0)$ be a $(-2)$-vector with
$$
\frac{a_1 d-rd_1}{r_1 d-rd_1}=\left(\frac{2d-r}{d-r}\right)^{(-1)^{i+1}}.
$$ 
Then $\langle w_i,v_1 \rangle>0$.
\end{lem}

\begin{NB}

\begin{proof}
We only treat the case where $i=1$.
We note that $w_1=(2(d-r),3(d-r)H,4d-5r)$.
We set $L:=({\Bbb Q}v+{\Bbb Q}v(F_1)) \cap H^*(X,{\Bbb Z})_{\alg}$.
%$L_1:=({\Bbb Q}(1,H,2)+{\Bbb Q}(2(d-r),3(d-r)H,4d-5r)) \cap H^*(X,{\Bbb Z})_{\alg}$.
Let $v_1=(r_1,d_1 H,a_1) \in L$ be a $(-2)$-vector such that
$d_1>0$.
Then $a_1=\frac{(2d-r)r_1-rd_1}{d-r}$.
If $r_1 \geq d_1$, then 
$$
a_1=\frac{(r_1-d_1)r}{d-r}+2r_1 \geq 2r_1 \geq 2d_1.
$$
Hence $d_1^2+1=r_1 a_1 \geq 2d_1^2$, which implies $v_1=(1,H,2)$.
In this case $\langle w_1,v_1 \rangle=3r-2d >0$.
If $r_1<d_1$, then
$\langle w_1,v_1 \rangle =(6(d-r)+2r)d_1-(8(d-r)+r)r_1>0$.
Therefore there is a $\sigma_{(0,t_0)}$-stable object $F$ with
$v(F)=w_1$, where $t_0=\sqrt{\frac{2d-r}{d-r}}$.
\end{proof}

We set $L_2:=({\Bbb Q}(2,H,1)+{\Bbb Q}(12-r, 9H,6)) \cap H^*(X,{\Bbb Z})_{\alg}$.
Let $v_1=(r_1,d_1 H,a_1) \in L_2$ be a $(-2)$-vector such that
$d_1>0$.
Then $r_1=\frac{(r+6)a_1-rd_1}{3}$.
If $a_1 \geq d_1$, then $r_1=(a_1-d_1)r/3+2a_1 \geq 2a_1$.
Hence $d_1^2+1=r_1 a_1 \geq 2d_1^2$, which implies $v_1=(2,H,1)$.
In this case $\langle (12-r,9H,6),v_1 \rangle=r-6 >0$.
If $a_1<d_1$, then
$\langle (12-r,9H,6),v_1 \rangle= (18+2r)d_1-(r+24)a_1>0$.
Therefore there is a $\sigma_{(0,t_0)}$-stable object $F$ with
$v(F)=(12-r,9H,6)$, where $t_0=\frac{3}{r+6}$.

\end{NB}

\begin{NB}

We set $L_1:=({\Bbb Q}(1,H,2)+{\Bbb Q}(4,6H,8-r)) \cap H^*(X,{\Bbb Z})_{\alg}$.
Let $v_1=(r_1,d_1 H,a_1) \in L_1$ be a $(-2)$-vector such that
$d_1>0$.
Then $a_1=\frac{(r+4)r_1-rd_1}{2}$.
If $r_1 \geq d_1$, then $a_1=(r_1-d_1)r/2+2r_1 \geq 2r_1$.
Hence $d_1^2+1=r_1 a_1 \geq 2d_1^2$, which implies $v_1=(1,H,2)$.
In this case $\langle (4,6H,8-r),v_1 \rangle=r-4 >0$.
If $r_1<d_1$, then
$\langle (4,6H,8-r),v_1 \rangle =(12+2r)d_1-(r+16)r_1>0$.
Therefore there is a $\sigma_{(0,t_0)}$-stable object $F$ with
$v(F)=(4,6H,8-r)$, where $t_0=\frac{r+4}{2}$.

We set $L_2:=({\Bbb Q}(2,H,1)+{\Bbb Q}(8-r, 6H,4)) \cap H^*(X,{\Bbb Z})_{\alg}$.
Let $v_1=(r_1,d_1 H,a_1) \in L_2$ be a $(-2)$-vector such that
$d_1>0$.
Then $r_1=\frac{(r+4)a_1-rd_1}{2}$.
If $a_1 \geq d_1$, then $r_1=(a_1-d_1)r/2+2a_1 \geq 2a_1$.
Hence $d_1^2+1=r_1 a_1 \geq 2d_1^2$, which implies $v_1=(2,H,1)$.
In this case $\langle (8-r,6H,4),v_1 \rangle=r-4 >0$.
If $a_1<d_1$, then
$\langle (8-r,6H,4),v_1 \rangle= (12+2r)d_1-(r+16)a_1>0$.
Therefore there is a $\sigma_{(0,t_0)}$-stable object $F$ with
$v(F)=(8-r,6H,4)$, where $t_0=\frac{2}{r+4}$.

\end{NB}

\begin{NB}
For $v=(3,4H,3)$,
$(1,H,2)$ and $(2,H,1)$ define totally semistable walls.
Hence $D_X \circ \Phi_{X \to X}^{{\cal E}}$ does not induce a birationl map,
where ${\cal E}=I_\Delta$ is the ideal sheaf of the diagonal $\Delta \subset X \times X$.

$\langle (1,H,2),(r,dH,r) \rangle=2d-3r<0$ and $\frac{2d-r}{d-r}>0$.
For a general $E \in {\cal M}_H(r,dH,r)$,
we have an exact sequence
$$
0 \to {\cal O}_X(H)^{\oplus (3r-2d)} \to E \to F \to 0,
$$
where
$F \in {\cal M}_H(2(d-r),3(d-r),4d-5r)$.

\end{NB}

\begin{rem}
For $v=(r,(r+1)H,r)$, $(1,H,2)$ and $(2,H,1)$ still give totally semistable wall, but the birational automorphism 
of $M_H(v)$
is slightly different. In this case, $\langle v,(1,H,1) \rangle=0$ and
defines a wall for the movable cone. So
$D_X \circ \Phi_{X \to X}^{{\cal E}} \circ \Psi$ induces a birational map 
acting trivially on the movable cone of $M_H(v)$.
\end{rem}

\begin{rem}
Assume that $v=(r,(r+3)H,r)$. Then a similar claim holds for $r \geq 7$.
Thus $\langle v,v_1 \rangle \leq 0$ iff $v_1=(1,H,2),(2,H,1)$ and
they define totally semistable walls.
In particular \eqref{eq:Psi} holds.
\end{rem}

\begin{NB}
Assume that $v=(r,(r+3)H,r)$.
We shall classfy $v_1=(r_1,d_1 H,a_1) \in D_v^\beta$ satisfying $\langle v_1^2 \rangle=-2$ and
$\langle v,v_1 \rangle \leq 0$.
By Lemma \ref{lem:r_1<0},
$r_1 d-r d_1>0$ and $a_1 d-a d_1>0$.
We note that
$\frac{d_1}{d}r=d_1-\frac{3d_1}{r+3} \geq d_1-3$.
Hence $r_1 \geq d_1-2$ and $a_1 \geq d_1-2$.
Assume that $d_1 \geq 8$.
If $r_1 \geq d_1+3$, then $(d_1+3)(d_1-2)>d_1^2+1$.
Therefore $r_1 \leq d_1+2$. We also get $a_1 \leq d_1+2$.
Since $\frac{d_1^2+1}{d_1+2}=d_1-2+\frac{5}{d_1+2}$ and
$\frac{d_1^2+1}{d_1+1}=d_1-1+\frac{2}{d_1+1}$ are not integers,
we get a contradiction.
So $d_1 \leq 7$.
If $d_1=7$, then $r_1 \geq d_1-2=5$ and $a_1 \geq d_1-2=5$.
Hence $\{r_1,a_1 \}=\{10,5 \}$.
In this case, $5>\frac{7r}{r+3}$ implies $r \leq 7$. 
We have $\langle v,v_1 \rangle=42-r>0$.
If $d_1=6$, then $r_1 a_1=37$. By $r_1,a_1 \geq d_1-2=4$, there is no chioce.
If $d_1=5$, then $r_1 a_1=26$. By $r_1,a_1 \geq d_1-2=3$, there is no choice. 
If $d_1=4$, then $r_1 a_1=17$. By $r_1,a_1 \geq d_1-2=2$, there is no choice.
If $d_1=3$, then $r_1 a_1=10$. If $2>\frac{3r}{r+3}$, then $r \leq 5$.
We have $\langle v,(2,3H,5) \rangle=18-r$.
If $1>\frac{3r}{r+3}$, then $r \leq 1$ and $\langle v,(1,3H,10) \rangle=18-5r>0$.
If $d_1=2$, then $r_1 a_1=5$. If $1>\frac{2r}{r+3}$, then $r \leq 2$.
If $d_1=1$, then $\{ r_1,a_1 \}=\{ 1,2 \}$.
Hence $\langle v,v_1 \rangle=6-r$.
 
If $r=4,5$, then there is no required $v_1$.
If $r=6$, then $\langle v,(1,H,2) \rangle=0$.
If $r \geq 7$, then $\langle v,v_1 \rangle \leq 0$ iff $v_1=(1,H,2),(2,H,1)$.

\end{NB}

\subsection{Stable sheaves parameterized by the Hilbert scheme of points.}\label{subsect:Hilb}
For the classification of stable sheaves, it is nice to express a general stable sheaves by simpler objects. 
For an abelian surface with Picard number 1,
if the moduli space is birationally equivalent to 
the moduli space of rank 1 stable sheaves, then
a general stable sheaves can be described by using semihomogeneous sheaves.
For a K3 surface, we cannot expect such a concrete description in general.
In this subsection, we give a partial result on this problem. 
Let $u_0,u_1$ be Mukai vectors such that
\begin{equation}
\langle u_0^2 \rangle=\langle u_1^2 \rangle=0,\; \langle u_0,u_1 \rangle=-1.
\end{equation}
Replacing $u_0$ by $-u_0$ if necessary, we assume that $u_0$ is a Mukai vector of 
a stable sheaf.
There are integers $p_1,p_2,q_1,q_2,s,t$ such that 
$s,t>0$, $st=n$, $p_2 q_1 t-p_1 q_2 s=\pm 1$ and
$$
u_0=(p_1^2 s,p_1 q_1 H,q_1^2 t),\; u_1=(p_2^2 t,p_2 q_2 H,p_2^2 s).
$$  
We set $X':=M_H(u_0)$.
Then $X'$ is a K3 surface and there is a universal family ${\cal E}$.
We assume that $p_2 q_1 t-p_1 q_2 s=1$.
Then $\frac{p_2 q_2}{p_2^2 t}<\frac{p_1 q_1}{p_1^2 s}$.
Since $\langle (u_0+u_1)^2 \rangle=-2$, there is a stable vector bundle $E_0$
with $v(E_0)=u_0+u_1$.

\begin{prop}
Assume that $t \geq p_1 s(p_1 s+2)$ and $s \ne 1$.
For a general homomorphism
\begin{equation}\label{eq:varphi}
\varphi:E_0 \to \oplus_{i=1}^{\ell+1} {\cal E}_{|X \times \{ x_i' \}},
\end{equation}
one of the following holds:
\begin{enumerate}
\item
$\varphi^{\vee}$ is injective and $\coker \varphi^{\vee} \in {\cal M}_H(u_1'-\ell u_0')$. 
Moreover a general $F \in {\cal M}_H(u_1'-\ell u_0')$ is given in this way, where 
$\varphi^{\vee}$ is the dual of $\varphi$ and
$u_i'=D_X(u_i)$. 
\item
$\varphi$ is injective and $\coker \varphi \in {\cal M}_H(\ell u_0-u_1)$.
Moreover a general $F \in {\cal M}_H(\ell u_0-u_1)$ is given in this way. 
\end{enumerate}
\end{prop}

\begin{proof}
Since $s,t>1$, ${\cal E}$ is a family of $\mu$-stable vector bundles on $X$.
By \cite{Y:Stability} or \cite{FM-duality},
${\cal E}$ is regarded as a universal family of $\mu$-stable vector bundles on $X'$. 
Replacing ${\cal E}$ if necessary, we may assume that
$\Phi_{X \to X'}^{{\cal E}^{\vee}}(u_1)=(1,0,0)$.
By $t \geq p_1 s(p_1 s+2)$, Theorem \ref{thm:n} implies
$\Phi_{X' \to X}^{{\cal E}}({\cal O}_{X'})$ is a stable sheaf with the Mukai vector
$u_0+u_1$.
Hence $E_0= \Phi_{X' \to X}^{{\cal E}}({\cal O}_{X'})$.
Then 
the complex \eqref{eq:varphi} is the Fourier-Mukai transform
$\Phi_{X' \to X}^{{\cal E}}$
of the complex
$$
{\cal O}_{X'} \to {\cal O}_Z,
$$ 
where $Z=\{x_1',x_2',...,x_{\ell+1}' \}$.
\begin{NB}
\begin{equation}
\begin{split}
\Hom(E_0,{\cal E}_{|X \times \{ x' \}})=& {\Bbb C}\\
\Ext^1(E_0,{\cal E}_{|X \times \{ x' \}})=&0\\
\Ext^2(E_0,{\cal E}_{|X \times \{ x' \}})=&0
\end{split}
\end{equation}
for all $x' \in X'$.
\end{NB} 
%Hence ${\cal F}:={\cal E}^{\vee}$ is also a family of $\mu$-stable vector bundles on $X$.
Since $t \geq p_1 s(p_1 s+2)$, we can apply Theorem \ref{thm:n} and theorem \ref{thm:a<0}.
Thus $\Phi_{X' \to X}^{{\cal E}}(I_Z)^{\vee}$ is a stable sheaf for a general $I_Z$ if 
$p_2^2 t-\ell p_1^2 s \geq 0$ and
$\Phi_{X' \to X}^{{\cal E}}(I_Z)[1]$ is a stable sheaf for a general $I_Z$
if $p_2^2 t-\ell p_1^2 s<0$.
Therefore our claim holds.
\end{proof}

\begin{rem}
Under the assumption $\Phi_{X \to X'}^{{\cal E}^{\vee}}(E_0)[2]={\cal O}_{X'}$,
if $\ell$ is sufficiently large, then Theorem \ref{thm:a<<0} implies
$\coker \varphi$ is stable for a general $\varphi$. 
For example, assume that $p_1=1$. Then $h^1({\cal E}_{|\{ x \} \times X'})=0$
for a general $x \in X$
(see \cite[Thm. 1.1 (2)]{CNY}). Then $\Phi_{X' \to X}^{{\cal E}}({\cal O}_{X'})^{\vee}$ is a sheaf, and hence
a rigid stable bundle by \cite{Mukai:Tata}.
Hence $\Phi_{X' \to X}^{{\cal E}}({\cal O}_{X'})=E_0$.
\end{rem}

\begin{NB}

$n=st$. $n \geq p_1 s(p_1 s+2)$.
We set
$X'=M_H(p_1^2 s,p_1 p_2 H,p_2^2 t)$ $(p_1,p_2>0)$, where
$p_2 q_1 t-sp_1 q_2=1$, that is
$q_2=\frac{p_2 q_1 t-1}{p_1 s}$.
Then there is a universal family ${\cal E}$ which is locally free.

\begin{equation}
A=
\begin{pmatrix}
p_2 \sqrt{t} & q_2 \sqrt{s}\\
p_1 \sqrt{s} & q_1 \sqrt{t}
\end{pmatrix}
\end{equation}
$\det A=p_2 q_1 t-p_1 q_2 s=1$.
$\phi_A$ corresponds to $\Phi_{X' \to X}^{\cal E}$.

$\phi_A((1,0,0))=$
$\phi_A((0,0,1))=(p_1^2 s,p_1 p_2 H,p_2^2 t)$.

\begin{equation}
A^{-1}=
\begin{pmatrix}
q_1 \sqrt{t} & -q_2 \sqrt{s}\\
-p_1 \sqrt{s} & p_2 \sqrt{t}
\end{pmatrix}
\end{equation}

$v({\cal E}_{|\{x \} \times X'})=(p_1^2 s,p_1 p_2 H',p_2^2 t)$.

$\Phi_{X' \to X}^{{\cal E}}({\cal O}_{X'}) \in {\cal M}_H(p_2^2 t+p_1^2 s, (p_2 q_2+p_1 q_1)H,q_2^2 s+q_1^2 t)$.

$v=(p_2^2 t,p_2 q_2 H,q_2^2 s)-\ell(p_1^2 s,p_1 q_1 H,q_1^2 t)$.
For a general $E \in {\cal M}_{H'}(v)$, we have an exact sequence
$$
0 \to E \to E_0 \to \oplus_{i=1}^{\ell+1} {\cal E}_{|\{ x_i \} \times X'} \to 0
$$
where
$E_0 \in {\cal M}_{H'}(p_2^2 t+p_1^2 s, (p_2 q_2+p_1 q_1)H',q_2^2 s+q_1^2 t)$.

We set $v=-(p_2^2 t,p_2 q_2 H,q_2^2 s)+\ell(p_1^2 s,p_1 q_1 H,q_1^2 t)$.
For a general $E \in {\cal M}_{H'}(v)$, we have an exact sequence
$$
0 \to E_0 \to \oplus_{i=1}^{\ell+1} {\cal E}_{|\{ x_i \} \times X'} \to E \to 0.
$$
Indeed $\Phi_{X \to X'}^{\cal E}=\Phi_{X \to X'}^{{\cal F}^{\vee}}$, where ${\cal F}={\cal E}^{\vee}$.

\end{NB}

\begin{NB}

\section{Appendix}

\subsection{The case where $X$ is an abelian surface.}
Let $X$ be an abelian surface and $H$ an ample divisor such that
$(H^2)=2n$ and
$\NS(X)={\Bbb Z}H$.
In the same way as in the proof of Theorem \ref{thm:n} and
\cite[Thm. 4.2]{aCM:K3}, we get the following theorem.

\begin{thm}\label{thm:abel}
Let $v=(r,dH,a)$ be a Mukai vector such that
$$
r \geq 0, \; pd-rq>0,\;
\langle v^2 \rangle \geq 0.
$$ 
Assume that $n \geq p^2$.
\begin{enumerate}
\item[(1)] 
If $\langle v,v_0 \rangle<0$, then
$D_{X'} \circ \Phi_{X \to X'}^{{\cal E}}(E) \in {\cal M}_{H'}(v')$
for a general $E \in {\cal M}_H(v) \setminus Z$, where  
$v'=D_{X'} \circ \Phi_{X \to X'}^{{\cal E}}(v)$ and $Z$ is a closed substack of ${\cal M}_H(v)$. 
Moreover if $n >p^2$, then $\dim Z \leq \dim {\cal M}_H(v)-2$
\item[(2)]
If $\langle v,v_0 \rangle \geq 0$, then
$\Phi_{X \to X'}^{{\cal E}}(E)[1] \in {\cal M}_{H'}(v'')$
for a general $E \in {\cal M}_H(v) \setminus Z$, where  
$v''=-\Phi_{X \to X'}^{{\cal E}}(v)$ and $Z$ is a closed substack of ${\cal M}_H(v)$. 
Moreover $\dim Z \leq \dim {\cal M}_H(v)-2$
%Assume that $a \geq 2$. 
%Then for a general $E \in M_H(v)$,
%we have the following claims.
%\begin{enumerate}
%\item
%$H^0(X,E) \otimes {\cal O}_X \to E$ is surjective,
%\item
%$H^1(X,E)=H^2(X,E)=0$
%for a general $E \in M_H(v)$
%\end{enumerate}
\end{enumerate}
\end{thm}

\begin{NB2}
In the proof of Theorem 4.2 in \cite{aCM:K3}, 
$d^\beta>d^\beta_1$ for an abelian surface.
Hence we see that $\langle v,v_1 \rangle >0$.
Therefore $\dim Z \leq \dim {\cal M}_H(v)-2$. 
\end{NB2}

If there is no isotropic Mukai vector $u$ with $\langle v,u \rangle=1$, then
there is no totally semistable wall. Hence $\Phi_{X \to X'}^{{\cal E}}(E)$ is stable for a general $E$.
Thus the claim is meaningful only if
there is isotropic Mukai vector $u$ with $\langle v,u \rangle=1$.
Moreover we can get a more precise condition
for the non-existence of totally semistable wall, with which
we get a similar result to Theorem \ref{thm:abel}, although we cannot estimate 
$\dim Z$.

\begin{prop}\label{prop:p>s,t}
If there is a totally semistable wall on $s=\frac{q}{p}$, then there is a pair of positive integers
$s,t$ such that $n=st$, $\gcd(s,t)=1$ and $p>s,t$.
\end{prop}

\begin{proof}
If there is a totally semistable wall on $s=\frac{q}{p}$, then
\cite[Cor. 5.13]{YY2} implies that
we have a decomposition
\begin{equation}\label{eq:v}
v=\delta(\ell_1 (p_1^2 s,p_1 q_1 H,q_1^2 t)-\ell_2 (p_2^2 t,p_2 q_2 H,q_2^2 s)),\;
p_1,p_2,q_1,q_2,s,t \in {\Bbb Z}
\end{equation}
such that $\delta=\pm 1$,
$\{\ell_1,\ell_2\}=\{1,\ell \}$,
$p_1 q_2 s-p_2 q_1 t=1$, $st=n$,
$p_1,p_2,s,t>0$ and
\begin{equation}\label{eq:pq}
\frac{q_1}{p_1 s}<\frac{q}{p}<\frac{q_2}{p_2 t}.
\end{equation}
Then
\begin{equation}\label{eq:p>s,t}
-p<p_1 s(p_2 qt-q_2 p)<
0<t p_2(p_1 q s-pq_1)<p.
\end{equation}
Hence $p>s,t$.
\end{proof}

\begin{cor}
\begin{enumerate}
\item
Assume that $n$ is a power of a prime.
If $n \geq p$, then there is no totally semistable wall on $s=\frac{q}{p}$.
\item
Assume that $p=2$.
Then there is no totally semistable wall on $s=\frac{q}{p}$ if $n \geq 2$.
\item
Assume that $p=3$.
Then there is no totally semistable wall on $s=\frac{q}{p}$
if $n \geq 3$.
\item
Assume that $p \geq 3$.
Then
there is no totally semistable wall on $s=\frac{q}{p}$
if $n>(p-1)(p-2)$.
\end{enumerate}
\end{cor}

\begin{proof}
Let $s,t$ be positive integers such that $n=st$,
$\gcd(s,t)=1$ and $s,t <p$. We may assume that $s>t$.
If $n$ is a power of a prime, then $p>s=n$.
If $p=2,3$, then
$s=t=1$ or $p=3$ and $(s,t)=(1,2)$.
Hence $n<p$.
Assume that $p \geq 3$. 
By $s \leq p-1$ and $t \leq p-2$, we get $n \leq (p-1)(p-2)$.
Therefore our claims follow from Proposition \ref{prop:p>s,t}.
\end{proof}

\begin{rem}
\begin{enumerate}
\item
If $n<p$, then there is a totally semistable wall on $s=\frac{q}{p}$.
We take $\beta=-\frac{1}{p}$.
Then
$v_1=(n,-H,1)$ defines a totally semistable wall for $v=(n\ell,-\ell H,\ell)$.
\item
If $n$ is not a power of prime, then
there is a totally semistable wall on $s=\frac{q}{p}$
even if $n \geq p$ in general.

Indeed for 
$n=6$, we set $v=\ell (2,H,3)-(3,2H,8)$.
Then $v_1=(2,H,3)$ defines a totally semistable wall on $s=\frac{3}{5}$.
\end{enumerate}
\end{rem}

\begin{rem}
For the totally semistable wall in the proof of Proposition \ref{prop:p>s,t},
$p_1,p_2<p$.
By $q_2=\frac{p_2 q_1 t+1}{p_1 s}$ and \eqref{eq:pq},
the choice of $q_1,q_2$ is finite.
Therefore there are finitely many choices of $v=(r,dH,a)$ with a totally semistable wall on $s=\frac{q}{p}$,
 if we fixed $r$ or $\langle v^2 \rangle$.
\end{rem}

\begin{ex}
Assume that $p=2$ and $n=1$.
If there is a totally semistable wall, then
$$
s=t=p_1=p_2=p_1 q_2 s-p_2 q_1 t=p_1 qs-pq_1=q_2 p-p_1 q_2 s=1.
$$
Hence $q=2q_1+1$ and
$q_2=q_1+1$.
Then
$$
v=(\ell-1,-H,-1)e^{q_1 H}, (\ell-1,\ell H,\ell)e^{q_1 H},\;(\ell>1)
$$
or $v=(0,H,1)e^{q_1 H}$.
\end{ex}

\begin{NB2}
$p=3$ and $n=1$.

If $(p_1,p_2)=(2,1)$, then
$$
v_1=(4,-2,1)e^{q_2 H},\; v_2=(1,0,0)e^{q_2 H}.
$$

If $(p_1,p_2)=(1,2)$, then
$$
v_1=(1,0,0)e^{q_1 H},\; v_2=(4,2H,1)e^{q_1 H}.
$$

If $(p_1,p_2)=(1,1)$, then
$$
v_1=(1,0,0)e^{q_1 H},\; v_2=(1,H,1)e^{q_1 H}.
$$

\end{NB2}

In \cite{Y:Stability}, we proved that semistability is preserved for all members of moduli space
if $d$ is larger than a constant depending on $r$, $d \mod r$ and
$\langle v^2 \rangle$. For a general member of moduli spaces,
we get a more simple condition.

\begin{NB2}
If $dn >\max \left\{\frac{1}{k^2}(kr'+(k-1)r)\ell,\ell \right\}$, then
there is no wall on $s=0$, where $k:=\gcd(r,d)>0$, $0 \leq r'<r$ and $rd'-r'd=-k$.
If $dn \leq \ell$, then there is a wall on $s=0$ in general.  
\end{NB2}

\begin{prop}
If $|d-\frac{q}{p}r| \geq p+\frac{r}{n}$, then 
there is no totally semistable wall on $s=\frac{q}{p}$.
\end{prop}

\begin{proof}
Assume that there is a totally semistable wall on $s=\frac{q}{p}$.
It is sufficient to prove that
$$
\left|d-\frac{q}{p}r \right| < p+\frac{r}{n}.
$$
We take the decomposition \eqref{eq:v}:
$$
(r,dH,a)=\delta(\ell_1 (p_1^2 s,p_1 q_1 H,q_1^2 t)-\ell_2 (p_2^2 t,p_2 q_2 H,q_2^2 s)).
$$
We first assume that $\delta=1$. Then
\begin{equation}
r=\ell_1 p_1^2 s-\ell_2 p_2^2 t,\;
d=\ell_1 p_1 q_1-\ell_2 p_2 q_2.
\end{equation}

We see that
\begin{equation}
d-\frac{q}{p}r=\ell_1 p_1 q_1-\ell_2 p_2 \frac{p_2 q_1 t+1}{p_1 s}-\frac{q}{p}r
=\left(\frac{q_1}{p_1 s}-\frac{q}{p} \right)r-\frac{\ell_2 p_2}{p_1 s}
\end{equation}
and
\begin{equation}
d-\frac{q}{p}r=\ell_1 p_1 \frac{q_2 p_1 s-1}{p_2 t}-\ell_2 p_2 q_2-\frac{q}{p}r
=\left(\frac{q_2}{p_2 t}-\frac{q}{p} \right)r-\frac{\ell_1 p_1}{p_2 t}.
\end{equation}

By \eqref{eq:pq}, 
\begin{equation}
0<\frac{q}{p}-\frac{q_1}{p_1 s}<\frac{1}{n p_1 p_2}, \; 
0<\frac{q_2}{p_2 t}-\frac{q}{p}<\frac{1}{n p_1 p_2}.
\end{equation}
Therefore
\begin{equation}
0>d-\frac{q}{p}r>-\frac{r}{n}-p.
\end{equation}

We next assume that $\delta=-1$, that is,
\begin{equation}
r=-\ell_1 p_1^2 s+\ell_2 p_2^2 t,\;
d=-\ell_1 p_1 q_1+\ell_2 p_2 q_2.
\end{equation}
In this case, we see that
\begin{equation}
d-\frac{q}{p}r=\ell_2 p_2 q_2-\ell_1 p_1 q_1-\frac{q}{p}r=
\left( \frac{q_1}{p_1 s}-\frac{q}{p}\right)r+\frac{\ell_2 p_2}{p_1 s}<\frac{\ell_2 p_2}{p_1 s}
\end{equation}
and
\begin{equation}
0<d-\frac{q}{p}r=\ell_2 p_2 q_2-\ell_1 p_1 q_1-\frac{q}{p}r=
\left( \frac{q_2}{p_2 t}-\frac{q}{p}\right)r+\frac{\ell_1 p_1}{p_2 t}<\frac{r}{p_1 p_2 n}+\frac{\ell_1 p_1}{p_2 t}.
\end{equation}

Therefore 
\begin{equation}
0<d-\frac{q}{p}r<\frac{r}{n}+p.
\end{equation}
\end{proof}

If $\langle v^2 \rangle$ is sufficiently larger than $p$, then
we can also get the following result.

\begin{prop}
Let $v=(r,dH,a)$ be a Mukai vector such that $r>0$ and 
$\langle v^2 \rangle=2\ell$.
Assume that $\ell \geq p^2+r$. Then there is no totally semistable wall on $s=\frac{q}{p}$.
\end{prop}

\begin{proof}
Assume that there is a totally semistable wall on $s=\frac{q}{p}$.
Then $r=\pm(\ell_1 p_1^2 s-\ell_2 p_2^2 t)$.

If $r=\ell_1 p_1^2 s-\ell_2 p_2^2 t$, then
$\ell p_1^2 s-p_2^2 t \geq \ell-p_2^2>r$ and
$p_1^2 s- \ell p_2^2 t < p^2-\ell \leq -r$, which contradicts 
our assumption $\ell_1 p_1^2 s-\ell_2 p_2^2 t = r$.

If $r=\ell_2 p_2^2 t-\ell_1 p_1^2 s$, then
$\ell p_2^2 t-p_1^2 s \geq \ell-p_1^2>r$ and
$p_2^2 t- \ell p_1^2 s < p^2-\ell \leq -r$,
which congtradicts our assumption
$r=\ell_2 p_2^2 t-\ell_1 p_1^2 s$.

Therefore our claim holds.
\end{proof}

\end{NB}

\end{document}